\documentclass[a4paper,12pt,reqno,english]{amsart}
\usepackage[utf8]{inputenc}
\usepackage[T1]{fontenc}
\usepackage{babel}

\usepackage{amssymb}

\usepackage{aliascnt}
\usepackage[bookmarks = false, colorlinks, citecolor = blue, urlcolor = blue]{hyperref}
\usepackage[titletoc, toc, title]{appendix}
\usepackage{tikz}

\numberwithin{equation}{section}
\numberwithin{figure}{section}

\theoremstyle{plain}
  \newtheorem{theorem}{Theorem}

\theoremstyle{plain}
  \newaliascnt{proposition}{theorem}
  \newtheorem{proposition}[proposition]{Proposition}
  \aliascntresetthe{proposition}

\theoremstyle{plain}
  \newaliascnt{lemma}{theorem}
  \newtheorem{lemma}[lemma]{Lemma}
  \aliascntresetthe{lemma}

\theoremstyle{plain}
  \newaliascnt{corollary}{theorem}
  \newtheorem{corollary}[corollary]{Corollary}
  \aliascntresetthe{corollary}

\theoremstyle{definition}

\theoremstyle{definition}
  \newtheorem{remark}{Remark}

\theoremstyle{definition}
  \newtheorem{notation}{Notation}

\theoremstyle{definition}

\theoremstyle{definition}

\addto\extrasenglish{
  
}

\date{}

\begin{document}

\title[Dolbeault-Dirac operators and the Parthasarathy formula]{Dolbeault-Dirac operators, quantum Clifford algebras and the Parthasarathy formula}

\author{Marco Matassa}%

\email{marco.matassa@gmail.com, mmatassa@math.uio.no}

\address{Department of Mathematics, University of Oslo, P.B. 1053 Blindern, 0316 Oslo, Norway.}

\begin{abstract}
We consider Dolbeault-Dirac operators on quantized irreducible flag manifolds as defined by Krähmer and Tucker-Simmons.
We show that, in general, these operators do not satisfy a formula of Parthasarathy-type.
This is a consequence of two results that we prove here: we always have quadratic commutation relations for the relevant quantum root vectors, up to terms in the quantized Levi factor; there are examples of quantum Clifford algebras where the commutation relations are not of quadratic-constant type.
\end{abstract}

\maketitle

\section{Introduction}

Dolbeault-Dirac operators on Kähler manifolds can be written, up to a scalar, in the form $D = \eth + \eth^* \in U(\mathfrak{g}) \otimes \mathrm{Cl}$.
In the conventions we adopt, the element $\eth \in U(\mathfrak{g}) \otimes \mathrm{Cl}$ can be identified with the adjoint of the Dolbeault operator $\bar{\partial}$.
Here $U(\mathfrak{g})$ is the enveloping algebra of $\mathfrak{g}$ and $\mathrm{Cl}$ is an appropriate Clifford algebra.
The class of Kähler manifolds contains that of irreducible flag manifolds.
Dolbeault-Dirac operators on quantized irreducible flag manifolds where originally defined in \cite{qflag}.
This definition was revisited and extended in \cite{qflag2}, where these operators are given in the form $D = \eth + \eth^* \in U_q(\mathfrak{g}) \otimes \mathrm{Cl}_q$.
Now $U_q(\mathfrak{g})$ is the quantized enveloping algebra of $\mathfrak{g}$, while $\mathrm{Cl}_q$ is the quantum Clifford algebra introduced in the cited paper.
One of the main results there is that $\eth^2 = 0$, as in the classical case.

This brings us to the third item in the title of this paper, the Parthasarathy formula \cite{partha} (we recommend \cite{dirac-book} for a textbook derivation).
This formula expresses the square of Dolbeault-Dirac operators in terms of quadratic Casimirs, up to multiples of the identity.
This readily allows to compute the spectra of such operators in terms of the representation theory of the corresponding Lie algebras.
It is an interesting question whether a formula of Parthasarathy-type also exists in the quantum setting.
One important application would be to define spectral triples on quantized irreducible flag manifolds, which was the main motivation in \cite{qflag}.
Indeed, it would allow to check the compact resolvent condition for Dolbeault-Dirac operators, which is an important requirement for a spectral triple \cite{con-book}.

Up to now, a quantum Parthasarathy formula is known to hold for projective spaces as a consequence of the results in \cite{dd-proj}, which generalize those obtained for low-dimensional cases in \cite{ddl-plane} and \cite{ds-pod}.
We should point out that the setup of the cited paper is different from the one we consider here.
The connection between the two approaches was later made in \cite{mat-proj}, where similar results are shown to hold.
It seems plausible that a Parthasarathy-type formula should hold for quantized irreducible flag manifolds, of which projective spaces are an example.
This expectation is motivated by the results of Heckenberger and Kolb in \cite{flagcalc1, flagcalc2}: they show that these spaces admit a canonical $q$-analogue of the de Rham complex, with the homogenous components having the same dimensions as in the classical case.
We stress that this is definitely not the case for general quantum spaces.

One of the main results of this paper is that a Parthasarathy-type formula does not hold for all quantized irreducible flag manifolds.
In order to state this result, we need to give a precise definition of what we mean by such a formula.
Recall that, for a Dolbeault-Dirac operator $D$, the classical Parthasarathy formula can be expressed as the identity $D^2 \sim C \otimes 1$. Here $C$ is the quadratic Casimir of $\mathfrak{g}$ and $\sim$ denotes equality up to terms in the Levi factor.
We need to consider a weaker formulation of this result, since in this form it does not even hold for the case of quantum projective spaces.
Clearly the formula should contain central elements of $U_q(\mathfrak{g})$, in order to make the connection with representation theory.
We are also allowed to neglect terms in the quantized Levi factor $U_q(\mathfrak{l})$: indeed these act as bounded operators on sections of the spinor bundle, hence they are not important for checking compactness of the resolvent of $D$.
The result then takes the following form.

\begin{theorem}
\label{thm:no-parthasarathy}
Let $D$ be a Dolbeault-Dirac operator corresponding to a quantized irreducible flag manifold.
Then there exists a flag manifold such that we do not have
\[
D^2 \sim \sum_i C_i \otimes T_i,
\]
where the elements $C_i \in U_q(\mathfrak{g})$ are assumed to be central and $T_i \in \mathrm{Cl}_q$.
Here the symbol $\sim$ denotes equality up to terms in the quantized Levi factor $U_q(\mathfrak{l})$.
\end{theorem}

The strategy of the proof is as follows.
We begin by deriving the commutation relations for the relevant quantum root vectors, namely those appearing in the definition of $D$.
These turn out to be quadratic for all quantized irreducible flag manifolds, see \autoref{thm:comm-rel}.
This result hinges on the fact that the radical roots take a very special form in the irreducible case.
Using this result we obtain a general expression for $D^2$.
Next, we show that the assumption $D^2 \sim \sum_i C_i \otimes T_i$, with the elements $C_i$ being central, implies that certain terms appearing in the expression for $D^2$ should vanish.
This in turn implies that we should have certain quadratic commutation relations in the quantum Clifford algebra.
Therefore it suffices to find one example where such relations do not hold.
The example that we consider is that of the Lagrangian Grassmannian $LG(2, 4)$.
After going through the necessary computations, we finally show in \autoref{prop:not-quadratic} that the relevant quadratic relations do not hold in this case.

The paper is organized as follows.
In \autoref{sec:notation} we give some background and fix notations and conventions.
In \autoref{sec:comm-rel} we derive commutation relations for the quantum root vectors.
In \autoref{sec:parthasarathy} we discuss the implications of these relations for a quantum Parthasarathy formula.
In \autoref{sec:braiding} we obtain the relations for the exterior algebras corresponding to the Lagrangian Grassmannian.
In \autoref{sec:quantum-clifford} we derive explicit formulae for the quantum Clifford algebra.
Finally in \autoref{sec:comm-clifford} show that we do not have quadratic relations in this algebra.
In \autoref{sec:rescaling} we collect some formulae related to various possible rescalings.

\section{Notations and conventions}
\label{sec:notation}

In this section we fix some notations and briefly review some facts about complex simple Lie algebras, parabolic subalgebras and quantized enveloping algebras.

\subsection{Parabolic subalgebras}

Let $\mathfrak{g}$ be a finite-dimensional complex simple Lie algebra with a fixed Cartan subalgebra $\mathfrak{h}$.
We denote by $\Delta(\mathfrak{g})$ the root system, by $\Delta^{+}(\mathfrak{g})$ the positive roots and by $\Pi = \{ \alpha_{1}, \cdots, \alpha_{r} \}$ the simple roots.
Denote by $a_{ij}$ the entries of the Cartan matrix and by $(\cdot, \cdot)$ the usual invariant bilinear form on $\mathfrak{h}^{*}$.
In particular, in the simply-laced case we have $(\alpha_{i}, \alpha_{j}) = a_{ij}$.
Let $S \subset \Pi$ be a subset of the simple roots. Then we set
$$
\Delta(\mathfrak{l}) = \mathrm{span}(S) \cap \Delta(\mathfrak{g}), \quad
\Delta(\mathfrak{u}_{+}) = \Delta^{+}(\mathfrak{g}) \backslash \Delta^{+}(\mathfrak{l}).
$$
In terms of these roots we define
$$
\mathfrak{l} = \mathfrak{h} \oplus \bigoplus_{\alpha \in \Delta(\mathfrak{l})} \mathfrak{g}_{\alpha}, \quad
\mathfrak{u}_{\pm} = \bigoplus_{\alpha \in \Delta(\mathfrak{u}_{+})} \mathfrak{g}_{\pm \alpha}, \quad
\mathfrak{p} = \mathfrak{l} \oplus \mathfrak{u}_{+}.
$$
It follows that $\mathfrak{l}$ and $\mathfrak{u}_{\pm}$ are Lie subalgebras of $\mathfrak{g}$.
We call $\mathfrak{p}$ the \emph{standard parabolic subalgebra} associated to $S$ (and omit $S$ from the notation). The subalgebra $\mathfrak{l}$ is reductive and is called the \emph{Levi factor} of $\mathfrak{p}$, while $\mathfrak{u}_{+}$ is a nilpotent ideal of $\mathfrak{p}$ called the \emph{nilradical}. We refer to the roots of $\Delta(\mathfrak{u}_{+})$ as the \emph{radical roots}.
We have the commutation relations $[\mathfrak{u}_{+}, \mathfrak{u}_{-}] \subset \mathfrak{l}$.

The adjoint action of $\mathfrak{p}$ on $\mathfrak{g}$ descends to an action on $\mathfrak{g} / \mathfrak{p}$. The decomposition $\mathfrak{g} = \mathfrak{u}_{-} \oplus \mathfrak{p}$ gives $\mathfrak{g} / \mathfrak{p} \cong \mathfrak{u}_{-}$ as $\mathfrak{l}$-modules.
We say that $\mathfrak{p}$ is of \emph{cominuscule type} if $\mathfrak{g} / \mathfrak{p}$ is a simple $\mathfrak{p}$-module. The following well-known result readily allows to classify all cominuscule parabolics.

\begin{proposition}
A parabolic subalgebra $\mathfrak{p}$ is cominuscule if and only if it corresponds to $S = \Pi \backslash \{\alpha_t\}$, where the simple root $\alpha_t$ appears with multiplicity $1$ in the highest root of $\mathfrak{g}$.
\end{proposition}

Moreover, it is clear from this result that all radical roots contain $\alpha_t$ with multiplicity $1$. The classification of cominuscole parabolics is reported in \autoref{tab:class-comin}, see for example \cite{schubert}.

\begin{table}[h]

\begin{tabular}{|c|c|c|}
\hline
Root system & Dynkin diagram & Nomenclature
\tabularnewline
\hline
\hline

$A_r$ &

\begin{tikzpicture}[
root/.style = {circle, draw = black, fill = white, thick, inner sep = 0pt, minimum size = 3mm},
rootlabel/.style = {inner sep = 0pt, scale = .8}
]

\draw (0,0) -- (1,0) node {};
\draw [dotted] (1,0) -- (2,0) node {};
\draw [dotted] (2,0) -- (3,0) node {};
\draw (3,0) -- (4,0) node {};

\node at (0, 0) [root] {};
\node at (1, 0) [root] {};
\node at (2, 0) [root, fill = black] {};
\node at (3, 0) [root] {};
\node at (4, 0) [root] {};

\node at (0,-.35) [rootlabel] {$1$};
\node at (1,-.35) [rootlabel] {$2$};
\node at (2,-.35) [rootlabel] {$k$};
\node at (3,-.35) [rootlabel] {$r - 1$};
\node at (4,-.35) [rootlabel] {$r$};

\end{tikzpicture}

&

Grassmannian $Gr(k, r)$

\tabularnewline
\hline

$B_r$ &

\begin{tikzpicture}[
root/.style = {circle, draw = black, fill = white, thick, inner sep = 0pt, minimum size = 3mm},
rootlabel/.style = {inner sep = 0pt, scale = .8}
]

\draw (0,0) -- (1,0) node {};
\draw [dotted] (1,0) -- (2,0) node {};
\draw (2,0) -- (3,0) node {};
\draw (3,.05) -- (4,.05) node {};
\draw (3,-.05) -- (4,-.05) node {};
\draw (3.4,.2) -- (3.6,0) node {};
\draw (3.4,-.2) -- (3.6,0) node {};

\node at (0, 0) [root, fill = black] {};
\node at (1, 0) [root] {};
\node at (2, 0) [root] {};
\node at (3, 0) [root] {};
\node at (4, 0) [root] {};

\node at (0,-.35) [rootlabel] {$1$};
\node at (1,-.35) [rootlabel] {$2$};
\node at (2,-.35) [rootlabel] {$r - 2$};
\node at (3,-.35) [rootlabel] {$r - 1$};
\node at (4,-.35) [rootlabel] {$r$};

\end{tikzpicture}

&

Odd dimensional quadric $\mathbb{Q}^{2r - 1}$

\tabularnewline
\hline

$C_r$ &

\begin{tikzpicture}[
root/.style = {circle, draw = black, fill = white, thick, inner sep = 0pt, minimum size = 3mm},
rootlabel/.style = {inner sep = 0pt, scale = .8}
]

\draw (0,0) -- (1,0) node {};
\draw [dotted] (1,0) -- (2,0) node {};
\draw (2,0) -- (3,0) node {};
\draw (3,.05) -- (4,.05) node {};
\draw (3,-.05) -- (4,-.05) node {};
\draw (3.6,.2) -- (3.4,0) node {};
\draw (3.6,-.2) -- (3.4,0) node {};

\node at (0, 0) [root] {};
\node at (1, 0) [root] {};
\node at (2, 0) [root] {};
\node at (3, 0) [root] {};
\node at (4, 0) [root, fill = black] {};

\node at (0,-.35) [rootlabel] {$1$};
\node at (1,-.35) [rootlabel] {$2$};
\node at (2,-.35) [rootlabel] {$r - 2$};
\node at (3,-.35) [rootlabel] {$r - 1$};
\node at (4,-.35) [rootlabel] {$r$};

\end{tikzpicture}

&

Lagrangian Grassmannian $LG(r, 2r)$

\tabularnewline
\hline

$D_r$ &

\begin{tikzpicture}[
root/.style = {circle, draw = black, fill = white, thick, inner sep = 0pt, minimum size = 3mm},
rootlabel/.style = {inner sep = 0pt, scale = .8}
]

\draw (0,0) -- (1,0) node {};
\draw [dotted] (1,0) -- (2,0) node {};
\draw (2,0) -- (3,.25) node {};
\draw (2,0) -- (3,-.25) node {};

\node at (0, 0) [root, fill = black] {};
\node at (1, 0) [root] {};
\node at (2, 0) [root] {};
\node at (3, .25) [root] {};
\node at (3, -.25) [root] {};

\node at (0,-.35) [rootlabel] {$1$};
\node at (1,-.35) [rootlabel] {$2$};
\node at (2,-.35) [rootlabel] {$r - 2$};
\node at (3.7,.25) [rootlabel] {$r - 1$};
\node at (3.4,-.25) [rootlabel] {$r$};

\end{tikzpicture}

& Even dimensional quadric $\mathbb{Q}^{2r - 2}$ 

\tabularnewline
\hline

$D_r$ &

\begin{tikzpicture}[
root/.style = {circle, draw = black, fill = white, thick, inner sep = 0pt, minimum size = 3mm},
rootlabel/.style = {inner sep = 0pt, scale = .8}
]

\draw (0,0) -- (1,0) node {};
\draw [dotted] (1,0) -- (2,0) node {};
\draw (2,0) -- (3,.25) node {};
\draw (2,0) -- (3,-.25) node {};

\node at (0,0) [root] {};
\node at (1,0) [root] {};
\node at (2,0) [root] {};
\node at (3,.25) [root, fill = black] {};
\node at (3,-.25) [root, fill = black] {};

\node at (0,-.35) [rootlabel] {$1$};
\node at (1,-.35) [rootlabel] {$2$};
\node at (2,-.35) [rootlabel] {$r - 2$};
\node at (3.7,.25) [rootlabel] {$r - 1$};
\node at (3.4,-.25) [rootlabel] {$r$};

\end{tikzpicture}

& Orthogonal Grassmannian $OG(r + 1, 2r + 2)$

\tabularnewline
\hline 

$E_6$ &

\begin{tikzpicture}[
root/.style = {circle, draw = black, fill = white, thick, inner sep = 0pt, minimum size = 3mm},
rootlabel/.style = {inner sep = 0pt, scale = .8}
]

\draw (0,0) -- (1,0) node {};
\draw (1,0) -- (2,0) node {};
\draw (2,0) -- (3,0) node {};
\draw (3,0) -- (4,0) node {};
\draw (2,0) -- (2,.75) node {};

\node at (0, 0) [root, fill = black] {};
\node at (1, 0) [root] {};
\node at (2, 0) [root] {};
\node at (3, 0) [root] {};
\node at (4, 0) [root, fill = black] {};
\node at (2,.75) [root] {};

\node at (0,-.35) [rootlabel] {$1$};
\node at (1,-.35) [rootlabel] {$2$};
\node at (2,-.35) [rootlabel] {$3$};
\node at (3,-.35) [rootlabel] {$4$};
\node at (4,-.35) [rootlabel] {$5$};
\node at (2.35,.75) [rootlabel] {$6$};

\end{tikzpicture}

&

Cayley plane $\mathbb{OP}^2$

\tabularnewline
\hline

$E_7$ &

\begin{tikzpicture}[
root/.style = {circle, draw = black, fill = white, thick, inner sep = 0pt, minimum size = 3mm},
rootlabel/.style = {inner sep = 0pt, scale = .8}
]

\draw (0,0) -- (1,0) node {};
\draw (1,0) -- (2,0) node {};
\draw (2,0) -- (3,0) node {};
\draw (3,0) -- (4,0) node {};
\draw (4,0) -- (5,0) node {};
\draw (2,0) -- (2,.75) node {};

\node at (0, 0) [root] {};
\node at (1, 0) [root] {};
\node at (2, 0) [root] {};
\node at (3, 0) [root] {};
\node at (4, 0) [root] {};
\node at (5, 0) [root, fill = black] {};
\node at (2,.75) [root] {};

\node at (0,-.35) [rootlabel] {$1$};
\node at (1,-.35) [rootlabel] {$2$};
\node at (2,-.35) [rootlabel] {$3$};
\node at (3,-.35) [rootlabel] {$4$};
\node at (4,-.35) [rootlabel] {$5$};
\node at (5,-.35) [rootlabel] {$6$};
\node at (2.35,.75) [rootlabel] {$7$};

\end{tikzpicture}

&

(Unnamed) $G_\omega(\mathbb{O}^3, \mathbb{O}^6)$

\tabularnewline
\hline

\end{tabular}

\medskip

\caption{Classification of cominuscole parabolics. The black node corresponds to the simple root $\alpha_t$. If there is more than one, then they are equivalent choices.}
\label{tab:class-comin}

\end{table}

\subsection{Quantized enveloping algebras}

We briefly review some facts about quantized enveloping algebras. General references for this topic are the books \cite{klsc}, \cite{chpr} and \cite{lus-book}.
With the previous conventions for complex simple Lie algebras, let $d_i = (\alpha_i, \alpha_i) / 2$. Let $q \in \mathbb{C}$ and define $q_i = q^{d_i}$.
The \emph{quantized universal enveloping algebra} $U_q(\mathfrak{g})$ is generated by the elements $E_i$, $F_i$, $K_i$, $K_i^{-1}$, for $1\le i\le r$ and with $r$ the rank of $\mathfrak{g}$, satisfying the relations
\begin{gather*}
K_i K_i^{-1} = K_i^{-1} K_i=1,\ \
K_i K_j = K_j K_i, \\
K_i E_j K_i^{-1} = q_i^{a_{ij}} E_j,\ \
K_i F_j K_i^{-1} = q_i^{-a_{ij}} F_j, \\
E_i F_j - F_j E_i = \delta_{ij} \frac{K_i - K_i^{-1}}{q_i - q_i^{-1}},
\end{gather*}
plus the quantum analogue of the Serre relations.
The Hopf algebra structure is defined by
\begin{gather*}
\Delta(K_i)=K_i\otimes K_i, \quad
\Delta(E_i)=E_i\otimes1+ K_i\otimes E_i, \quad
\Delta(F_i)=F_i\otimes K_i^{-1}+1\otimes F_i, \\
S(K_{i}) = K_{i}^{-1}, \quad
S(E_{i}) = - K_{i}^{-1} E_{i}, \quad
S(F_{i}) = - F_{i} K_{i}, \quad
\varepsilon(K_i)=1, \quad
\varepsilon(E_i)=\varepsilon(F_i)=0.
\end{gather*}
For $q \in \mathbb{R}$, the \emph{compact real form} of $U_q(\mathfrak{g})$ is defined by
\[
K_{i}^{*} = K_{i}, \quad
E_{i}^{*} = K_{i} F_{i}, \quad
F_{i}^{*} = E_{i} K_{i}^{-1}.
\]

Let $\mathfrak{l}$ be the Levi factor corresponding to a parabolic subalgebra of $\mathfrak{g}$ defined by $S \subset \Pi$. Then the quantized enveloping algebra of the Levi factor is defined as
\[
U_q(\mathfrak{l}) = \{ \textrm{subalgebra of $U_q(\mathfrak{g})$ generated by $K_i^{\pm 1}$ and $E_j, F_j$ with $j \in S$} \}.
\]
This definition of the quantized Levi factor appears for example in \cite[Section 4]{quantum-flag}.

\subsection{Quantum root vectors}

Fix a reduced decomposition $w_0 = s_{i_1} \cdots s_{i_N}$ of the longest word of the Weyl group of $\mathfrak{g}$. Here $s_i$ is the reflection corresponding to $\alpha_i$. It is well known that all the positive roots can be obtained as $\beta_k = s_{i_1} \cdots s_{i_{k - 1}} (\alpha_{i_k})$ for $k = 1, \cdots, N$.

Now let $T_i$ be the Lusztig automorphisms.
The quantum root vectors are then defined by $E_{\beta_k} = T_{i_1} \cdots T_{i_{k - 1}} (E_{i_k})$ for $k = 1, \cdots, N$. They depend on the choice of the reduced decomposition of $w_0$. Similarly the quantum root vectors corresponding to the negative roots are defined by $F_{\beta_k} = T_{i_1} \cdots T_{i_{k - 1}} (F_{i_k})$ for $k = 1, \cdots, N$.

\section{Commutation relations}
\label{sec:comm-rel}

In this section we will discuss the commutation relations between the quantum roots vectors $E_\xi$ and $E_{\xi^\prime}^*$, where $\xi$ and $\xi^\prime$ are radical roots coming from some cominuscule parabolic subalgebra.
The main result of this section is \autoref{thm:comm-rel}, which shows that we have quadratic commutation relations, up to terms in the quantized Levi factor.

\subsection{Nilradical and adjoint action}

It is well-known that the irreducible representations of $U_q(\mathfrak{g})$ essentially coincide with those of $U(\mathfrak{g})$, when $q$ is not a root of unity.
Hence there exists a $U_q(\mathfrak{l})$-module corresponding to the classical nilradical, which we denote by $\mathfrak{u}_+$.
Suppose furthermore that $\mathfrak{u}_+$ comes from a \emph{cominuscule} parabolic subalgebra $\mathfrak{p}$.
Then it follows from the results of \cite{zwi} that $\mathfrak{u}_+$ can be identified with a certain subspace of $U_q(\mathfrak{g})$.

Let us briefly review this result.
First of all, recall that $U_q(\mathfrak{g})$ acts on itself by the adjoint action, which is defined by $X \triangleright Y = X_{(1)} Y S(X_{(2)})$, where as usual we use Sweedler's notation.
Consider now the subspace of $U_q(\mathfrak{g})$ spanned by the quantum root vectors $\{ E_\xi \}_\xi$.
These depend on the choice of decomposition of the longest word of the Weyl group.
As in \cite[Section 5.1]{zwi} we assume that it has a certain natural factorization (this won't be very important in the following, so we omit the details).
The result that we need is part of
\cite[Main Theorem 5.6]{zwi}, although stated in a slightly different language.

\begin{proposition}[Zwicknagl]
The vector space spanned by the quantum root vectors $\{ E_\xi \}_\xi$, together with the adjoint action, is isomorphic to $\mathfrak{u}_+$ as a $U_q(\mathfrak{l})$-module.
\end{proposition}

An immediate consequence is the following.

\begin{corollary}
\label{lem:adjoint-action}
Let $\xi \in \Delta(\mathfrak{u}_+)$ and $\alpha_k \in \Delta(\mathfrak{l})$. Then we have
\[
E_{k} \triangleright E_\xi = c_{k, \xi} E_{\xi + \alpha_k}, \quad F_{k} \triangleright E_\xi = c^\prime_{k, \xi} E_{\xi - \alpha_k},
\]
where $c_{k, \xi} = 0$ if $\xi + \alpha_k \notin \Delta(\mathfrak{u}_+)$ and $c^\prime_{k, \xi} = 0$ if $\xi - \alpha_k \notin \Delta(\mathfrak{u}_+)$.
\end{corollary}

In order to obtain commutation relations between $E_\xi$ and $E_{\xi^\prime}^*$, we shall also need the action of $U_q(\mathfrak{l})$ acts on the latter elements.
This easily follows from the previous result.

\begin{lemma}
The vector space spanned by the elements $\{E_{\xi}^*\}_\xi$ is invariant under the action of $U_q(\mathfrak{l})$.
In particular we have $E_{k} \triangleright E_{\xi}^* = -q^{-(\alpha_{k}, \alpha_{k})}(F_{k} \triangleright E_{\xi})^*$ for $\alpha_k \in \Delta(\mathfrak{l})$.
\end{lemma}

\begin{proof}
This easily follows from the general identity $X \triangleright Y^* = (S(X)^* \triangleright Y)^*$ and the fact that $U_q(\mathfrak{l})$ is closed under the antipode and the involution.
Otherwise one can proceed by direct computation.
The explicit action follows from our conventions for $U_q(\mathfrak{g})$.
\end{proof}

\subsection{Commutation relations}

We now proceed to derive the basic case of our commutation relations.
We are only interested in obtaining these relations modulo terms in the quantized Levi factor $U_q(\mathfrak{l})$.
For this reason we introduce the following notation.

\begin{notation}
For $X, Y \in U_q(\mathfrak{g})$ we write $X \sim Y$ if $X = Y + Z$ with $Z \in U_q(\mathfrak{l})$.
\end{notation}

We denote by $\alpha_t$ the unique simple radical root corresponding to a cominuscule parabolic.

\begin{lemma}
\label{lem:basic-comm}
Let $\xi \in \Delta(\mathfrak{u}_+)$. We have the commutation relations
\[
E_t E_\xi^* - q^{- (\alpha_t, \xi)} E_\xi^* E_t \sim 0.
\]
\end{lemma}

\begin{proof}
It is enough to prove that $[E_{\xi},  F_{t}] \sim 0$.
Indeed, as we will show below, we have the identity $E_t E_\xi^* - q^{-(\alpha_t, \xi)} E_\xi^* E_t = q^{-(\alpha_t, \xi)} [E_\xi, F_t]^* K_t$, hence the claim follows by observing that $K_t \in U_{q}(\mathfrak{l})$ and $U_{q}(\mathfrak{l})$ is invariant under $*$.
To show this identity let us consider
\[
[E_\xi, F_t]^* K_t = (K_t E_\xi F_{t} - K_t F_t E_\xi)^* = (q^{(\alpha_t, \xi)} E_\xi K_t F_t - K_t F_t E_\xi)^*.
\]
In the last step we have used the relation $K_t E_\xi = q^{(\alpha_t, \xi)} E_\xi K_t$. But then we have
\[
q^{-(\alpha_t, \xi)} [E_\xi, F_t]^* K_t
= q^{-(\alpha_t, \xi)} (q^{(\alpha_t, \xi)} E_\xi E_t^* - E_t^* E_\xi)^*
= E_t E_\xi^* - q^{-(\alpha_t, \xi)} E_\xi^* E_t.
\]

Now we will show that $[E_\xi, F_t] \sim 0$.
Write $\xi = \sum_{i = 1}^r c_i \alpha_i$. Then, using the grading by the root lattice, we conclude that $E_\xi$ is a sum of monomials of the form $E_{\sigma(1)}^{c_{\sigma(1)}} \cdots E_{\sigma(r)}^{c_{\sigma(r)}}$, where $\sigma$ denotes a permutation of $\{ 1, \cdots, r \}$.
Below we will consider in detail only the term $E_{1}^{c_{1}} \cdots E_{r}^{c_{r}}$. Recall that for any radical root $\xi$ the simple root $\alpha_t$ appears with multiplicity one, that is $c_t = 1$.
Since we have $[E_i, F_j] = 0$ for $i \neq j$, it follows that
\[
\begin{split}
[E_1^{c_1} \cdots E_r^{c_r}, F_t]
 & = E_1^{c_1} \cdots E_{t - 1}^{c_{t - 1}} [E_t, F_t] E_{t + 1}^{c_{t + 1}} \cdots E_r^{c_r}\\
 & = E_1^{c_1} \cdots E_{t - 1}^{c_{t - 1}} \frac{K_t - K_t^{-1}}{q_t - q_t^{-1}} E_{t + 1}^{c_{t + 1}} \cdots E_r^{c_r} \in U_{q}(\mathfrak{l}).
\end{split}
\]
The case of a general term $E_{\sigma(1)}^{c_{\sigma(1)}}  \cdots E_{\sigma(r)}^{c_{\sigma(r)}}$ is completely analogous.
\end{proof}

We are now ready to prove the main result of this section.

\begin{theorem}
\label{thm:comm-rel}
Let $\xi, \xi^\prime \in \Delta(\mathfrak{u}_+)$.
We have the commutation relations
\[
E_\xi E_{\xi^\prime}^* - q^{-(\xi, \xi^\prime)} E_{\xi^\prime}^* E_\xi \sim \sum_{\eta, \eta^\prime} c_{\xi, \xi^\prime}^{\eta, \eta^\prime} E_\eta^* E_{\eta^\prime},
\]
where in the sum we have the condition $\mathrm{ht}(\eta^\prime) < \mathrm{ht}(\xi)$.
\end{theorem}

\begin{proof}
We will proceed by induction over the height of $\xi$.
In the case $\mathrm{ht}(\xi) = 1$ we have only one radical root, namely the simple root $\alpha_t$.
Hence the result follows from \autoref{lem:basic-comm}.
We will now assume that the claim is true for all $\xi$ with $\mathrm{ht}(\xi) = n$.

We know that any radical root of height $n + 1$ can be written in the form $\xi + \alpha_k$, for some simple root $\alpha_k$ with $k \neq t$.
Our first step will be to obtain the commutation relation for the elements $E_{\xi + \alpha_k}$ and $E_{\xi^\prime}^*$ from that of the elements $E_\xi$ and $E_{\xi^\prime}^*$, using the adjoint action.
It is convenient at this point to introduce the $q$-commutator notation
\[
[E_\xi, E_{\xi^\prime}^*]_q = E_\xi E_{\xi^\prime}^* - q^{-(\xi, \xi^\prime)} E_{\xi^\prime}^* E_\xi.
\]
Acting with the element $E_k$ on this expression we get
\[
\begin{split}
E_k \triangleright [E_\xi, E_{\xi^\prime}^*]_q
 & = (E_k \triangleright E_\xi) E_{\xi^\prime}^* + (K_k \triangleright E_\xi) (E_k \triangleright E_{\xi^\prime}^*)\\
 & - q^{-(\xi, \xi^\prime)} (E_k \triangleright E_{\xi^\prime}^*) E_\xi - q^{-(\xi, \xi^\prime)} (K_k \triangleright E_{\xi^\prime}^*) (E_k \triangleright E_\xi).
\end{split}
\]
Using $K_k \triangleright E_\xi = q^{(\xi, \alpha_k)} E_\xi$ and $K_k \triangleright E_\xi^* = q^{-(\xi, \alpha_k)} E_\xi^*$ we rewrite this as
\[
\begin{split}
E_k \triangleright [E_\xi, E_{\xi^\prime}^*]_q
 & = (E_k \triangleright E_\xi) E_{\xi^\prime}^* - q^{-(\xi + \alpha_k, \xi^\prime)} E_{\xi^\prime}^* (E_k \triangleright E_\xi)\\
 & + q^{(\xi, \alpha_k)} E_\xi (E_k \triangleright E_{\xi^\prime}^*) - q^{-(\xi, \xi^\prime)} (E_k \triangleright E_{\xi^\prime}^*) E_\xi.
\end{split}
\]
Then using the action of $E_k$ as in \autoref{lem:adjoint-action} we obtain
\[
\begin{split}
E_k \triangleright [E_\xi, E_{\xi^\prime}^*]_q
 & = c_{k, \xi} (E_{\xi + \alpha_k} E_{\xi^\prime}^* - q^{-(\xi + \alpha_k, \xi^\prime)} E_{\xi^\prime}^* E_{\xi + \alpha_k})\\
 & - c_{k, \xi^\prime}^\prime q^{(\xi - \alpha_k, \alpha_k)} (E_\xi E_{\xi^\prime - \alpha_k}^* - q^{-(\xi, \xi^\prime + \alpha_k)} E_{\xi^\prime - \alpha_k}^* E_\xi).
\end{split}
\]
We are assuming that $\xi + \alpha_k \in \Delta(\mathfrak{u}_+)$, which guarantees that $c_{k, \xi} \neq 0$.
Hence we can divide by this factor and arrive at the identity
\[
[E_{\xi + \alpha_k}, E_{\xi^\prime}^*]_q = c_{k, \xi}^{-1} E_k \triangleright [E_\xi, E_{\xi^\prime}^*]_q + c_{k, \xi}^{-1} c_{k, \xi^\prime}^\prime q^{(\xi - \alpha_k, \alpha_k)} (E_\xi E_{\xi^\prime - \alpha_k}^* - q^{-(\xi, \xi^\prime + \alpha_k)} E_{\xi^\prime - \alpha_k}^* E_\xi).
\]
We have finally expressed $[E_{\xi + \alpha_k}, E_{\xi^\prime}^*]_q$ in terms of $[E_\xi, E_{\xi^\prime}^*]_q$, plus an additional term.
Notice that this second term does not have the form of a $q$-commutator.

To prove our claim we have to analyze these two terms. Let us start with the first one.
Since $\xi$ has height $n$, we can use the induction hypothesis and write
\[
[E_\xi, E_{\xi^\prime}^*]_q \sim \sum_{\eta, \eta^{\prime}} c_{\xi, \xi^\prime}^{\eta, \eta^\prime} E_\eta^* E_{\eta^\prime},
\]
where we have the condition $\mathrm{ht}(\eta^\prime) < \mathrm{ht}(\xi)$.
Then we compute
\[
E_k \triangleright (E_\eta^* E_{\eta^\prime}) = (E_k \triangleright E_\eta^*) E_{\eta^\prime} + (K_k \triangleright E_\eta^*) (E_k \triangleright E_{\eta^\prime}).
\]
The first summand is either zero or proportional to $E_{\eta - \alpha_k}^* E_{\eta^\prime}$ if $\eta - \alpha_k \in \Delta(\mathfrak{u}_+)$. Similarly the second summand is either zero or proportional to $E_{\eta}^* E_{\eta^\prime + \alpha_k}$ if $\eta^\prime + \alpha_k \in \Delta(\mathfrak{u}_+)$.
As a result, we can write again an identity of the form
\[
E_k \triangleright [E_\xi, E_{\xi^\prime}^*]_q \sim \sum_{\eta, \eta^{\prime}} b_{\xi, \xi^\prime}^{\eta, \eta^\prime} E_\eta^* E_{\eta^\prime},
\]
but now we have the condition $\mathrm{ht}(\eta^\prime) \leq \mathrm{ht}(\xi)$. Indeed, the sum on the right hand side may contain the term $E_\eta^* E_{\eta^\prime + \alpha_k}$ and we have the relation $\mathrm{ht}(\eta^\prime + \alpha_k) = \mathrm{ht}(\eta^\prime)+ 1$.

Now let us consider the second term in our expression for $[E_{\xi + \alpha_k}, E_{\xi^\prime}^*]_q$.
As for the first one, we can use the induction hypothesis to write
\[
E_\xi E_{\xi^\prime - \alpha_k}^* \sim q^{-(\xi, \xi^\prime - \alpha_k)} E_{\xi^\prime - \alpha_k}^* E_\xi + \sum_{\eta, \eta^\prime} c_{\xi, \xi^\prime - \alpha_k}^{\eta, \eta^\prime} E_\eta^* E_{\eta^\prime},
\]
with the condition $\mathrm{ht}(\eta^\prime) < \mathrm{ht}(\xi)$.
It follows that
\[
E_\xi E_{\xi^\prime - \alpha_k}^* - q^{-(\xi, \xi^\prime + \alpha_k)} E_{\xi^\prime - \alpha_k}^* E_\xi \sim (q^{-(\xi, \xi^\prime - \alpha_k)} - q^{-(\xi, \xi^\prime + \alpha_k)}) E_{\xi^\prime - \alpha_k}^* E_\xi + \sum_{\eta, \eta^\prime} c_{\xi, \xi^\prime - \alpha_k}^{\eta, \eta^\prime} E_\eta^* E_{\eta^\prime}.
\]
Notice that the term in parentheses does not vanish unless $q = 1$, in general.
Nevertheless, we can always rewrite this expression in the form
\[
E_\xi E_{\xi^\prime - \alpha_k}^* - q^{-(\xi, \xi^\prime + \alpha_k)} E_{\xi^\prime - \alpha_k}^* E_\xi \sim \sum_{\eta, \eta^{\prime}} \tilde{b}_{\xi, \xi^\prime - \alpha_k}^{\eta, \eta^\prime} E_\eta^* E_{\eta^\prime},
\]
but now we have the condition $\mathrm{ht}(\eta^\prime) \leq \mathrm{ht}(\xi)$. Indeed this sum contains the term $E_{\xi^\prime - \alpha_k}^* E_\xi$.

Finally, putting all these results together and relabeling our coefficients, we get
\[
[E_{\xi + \alpha_k}, E_{\xi^\prime}^*]_q \sim \sum_{\eta, \eta^{\prime}} c_{\xi + \alpha_k, \xi^\prime}^{\eta, \eta^\prime} E_\eta^* E_{\eta^\prime},
\]
with the condition $\mathrm{ht}(\eta^\prime) \leq \mathrm{ht}(\xi)$, that is $\mathrm{ht}(\eta^\prime) < \mathrm{ht}(\xi + \alpha_k)$.
This concludes the proof.
\end{proof}

\begin{remark}
It is shown in \cite{mat-comm} that the commutation relations between the elements $E_\xi$ and $F_{\xi^\prime}$ take a simpler form, at least for certain reduced decompositions. Namely we have $[E_\xi, F_{\xi^\prime}] \sim 0$, in perfect analogy with the classical setting.
\end{remark}

Later on we will need the following variant of the above result.

\begin{corollary}
\label{cor:comm-rel}
Let $\xi, \xi^\prime \in \Delta(\mathfrak{u}_+)$.
We have the commutation relations
\[
S^{-1}(E_{\xi^\prime})^* S^{-1}(E_\xi) - q^{-(\xi, \xi^\prime)} S^{-1}(E_\xi) S^{-1}(E_{\xi^\prime})^* \sim \sum_{\eta, \eta^\prime} c_{\xi, \xi^\prime}^{\eta^\prime, \eta} S^{-1}(E_{\eta^\prime}) S^{-1}(E_\eta)^*,
\]
where in the sum we have the condition $\mathrm{ht}(\eta^\prime) < \mathrm{ht}(\xi)$.
\end{corollary}

\begin{proof}
If we apply the anti-homomorphism $S^{-1}$ to \autoref{thm:comm-rel} we obtain
\[
S^{-1}(E_{\xi^\prime}^*) S^{-1}(E_\xi) - q^{-(\xi, \xi^\prime)} S^{-1}(E_\xi) S^{-1}(E_{\xi^\prime}^*) \sim \sum_{\eta, \eta^\prime} c_{\xi, \xi^\prime}^{\eta, \eta^\prime} S^{-1}(E_{\eta^\prime}) S^{-1}(E_\eta^*),
\]
From the general property $S^{-1} \circ * = * \circ S$ it follows that $S^{-1}(E_{\xi^\prime}^*) = S(E_{\xi^\prime})^*$.
Now for any element $X \in U_q(\mathfrak{g})$ we have the identity $S^2(X) = K_{2 \rho} X K_{2 \rho}^{-1}$, where $\rho$ is the half-sum of the positive roots. Hence $S^2(E_{\xi^\prime}) = q^{(\xi^\prime, 2 \rho)} E_\xi^\prime$.
We conclude that
\[
S^{-1}(E_{\xi^\prime}^*) = (S^{-1}(S^2(E_{\xi^\prime})))^* = q^{(\xi, 2 \rho)} S^{-1}(E_{\xi^\prime})^*.
\]
Upon relabeling the coefficients, we obtain the result.
\end{proof}

\section{On the Parthasarathy formula}
\label{sec:parthasarathy}

In this section we discuss the implications of the commutation relations of \autoref{thm:comm-rel}  for a quantum version of the Parthasarathy formula.
We start with a brief review of the classical case, from an appropriate perspective, and then move to the quantum setting.

\subsection{Dolbeault–Dirac operators}
A pair $(\mathfrak{g}, \mathfrak{p})$, where $\mathfrak{g}$ is a complex semisimple Lie algebra and $\mathfrak{p}$ is a parabolic Lie subalgebra, gives an infinitesimal description of the complex manifold $G / P$.
Here $G$ is the (connected, simply-connected) Lie group with Lie algebra $\mathfrak{g}$ and $P$ is the subgroup corresponding to $\mathfrak{p}$.
These spaces are referred to as \emph{generalized flag manifolds}.
Being complex manifolds, we have Dolbeault operators acting on differential forms.

Since the focus of this paper is on quantum Dolbeault-Dirac operators, in this discussion of the classical setting we will adopt a point of view which is well-suited for quantization.
Following \cite{qflag2}, we will define an element $\eth \in U(\mathfrak{g}) \otimes \mathrm{Cl}$,
where $U(\mathfrak{g})$ is the enveloping algebra of $\mathfrak{g}$ and $\mathrm{Cl}$ is the (complex) Clifford algebra of $\mathfrak{u}_{+} \oplus \mathfrak{u}_{-}$.
We will consider the Clifford algebra as represented on the exterior algebra $\Lambda(\mathfrak{u}_{+})$, the representation being given in terms of exterior and interior multiplication.
The operator $\eth$ to be defined below turns out to coincide, once the appropriate identifications are made, with the adjoint of the Dolbeault operator $\bar{\partial}: \Omega^{(0, k)} \to \Omega^{(0, k + 1)}$, see the discussion in \cite[Section 7]{qflag2}.
The adjoint is taken with respect to an invariant Hermitian inner product on $\Lambda(\mathfrak{u}_+)$.

Let us see how this operator is defined.
Recall that $\mathfrak{u}_+$ and $\mathfrak{u}_-$ are dual as $\mathfrak{l}$-modules, where $\mathfrak{l}$ is the Levi factor.
Pick a weight basis $\{v_i\}_i$ of $\mathfrak{u}_+$ and denote by $\{w_i\}_i$ the dual basis of $\mathfrak{u}_-$.
We can identify these bases with root vectors of $\mathfrak{g}$, in which case the dual pairing becomes the Killing form.
Write $\{E_{\xi_i}\}_i \in U(\mathfrak{g})$ for the root vectors corresponding to the radical roots and $\gamma_-(w)$ for interior multiplication by $w \in \mathfrak{u}_-$ on $\Lambda(\mathfrak{u}_+)$.
With these preparations, we define
\[
\eth = - \sum_{\xi_i \in \Delta(\mathfrak{u}_+)} E_{\xi_i} \otimes \gamma_-(w_i) \in U(\mathfrak{g}) \otimes \mathrm{Cl}.
\]
It does not depend on the choice of bases, as its definition involves dual bases.
The minus sign is a matter of convention.
Finally we define the \emph{Dolbeault–Dirac operator} as
\[
D = \eth + \eth^* \in U(\mathfrak{g}) \otimes \mathrm{Cl}.
\]
Here $*$ denotes the adjoint, which algebraically is implemented as a $*$-structure on $U(\mathfrak{g}) \otimes \mathrm{Cl}$. More specifically, on the first factor it corresponds to the $*$-structure on $U(\mathfrak{g})$ induced by the compact real form of $\mathfrak{g}$, while on the second factor it comes from the choice of a Hermitian inner product on $\Lambda(\mathfrak{u}_{+})$.
It can be seen that $D$ acts, up to a scalar, as the Dolbeault–Dirac operator on $G/Q$ formed with respect to the canonical $\mathrm{spin}^{c}$ structure.

\subsection{The Parthasarathy formula}

We will now discuss a simple way to compute the square of $D$ on a generalized flag manifold $G / P$ (strictly speaking $G / Q$, in our conventions).

Let $\{v_i\}_i$ and $\{w_i\}_i$ be as above, but furthermore we require that $\{v_i\}_i$ is orthonormal with respect to the invariant Hermitian inner product on $\mathfrak{u}_+$.
Then it extends to an orthonormal basis of $\Lambda(\mathfrak{u}_+)$ and moreover the adjoint of $\gamma_-(w_i)$, the operator of interior multiplication by $w_i \in \mathfrak{u}_-$, is given by $\gamma_+(v_i)$, the operator of exterior multiplication by $v_i \in \mathfrak{u}_+$.
With these conventions, we can write the Dolbeault-Dirac operator as
\[
D = \eth + \eth^* = \sum_{i} E_{\xi_i} \otimes \gamma_-(w_i) + \sum_{i} F_{\xi_i} \otimes \gamma_+(v_i).
\]
We have $\eth^2 = 0$, as a consequence of $\bar{\partial}^2 = 0$. Then we obtain
\[
D^2 = \sum_{i, j} E_{\xi_i} F_{\xi_j} \otimes \gamma_-(w_i) \gamma_+(v_j) + \sum_{i, j} F_{\xi_i} E_{\xi_j} \otimes \gamma_+(v_i) \gamma_-(w_j).
\]
For $A, B \in U(\mathfrak{g})$ we write $A \sim B$ if $A = B + C$ for some $C \in U(\mathfrak{l})$.
Recall that we have the commutation relation $[\mathfrak{u}_{+}, \mathfrak{u}_{-}] \subset \mathfrak{l}$. Therefore $[E_{\xi_i}, F_{\xi_j}] \in \mathfrak{l}$, so that we can write $F_{\xi_j} E_{\xi_i} \sim E_{\xi_i} F_{\xi_j}$.
Therefore after relabeling we get
\[
D^2 \sim \sum_{i, j} E_{\xi_i} F_{\xi_j} \otimes (\gamma_-(w_i) \gamma_+(v_j) + \gamma_+(v_j) \gamma_-(w_i)).
\]
We have the following commutation relations between interior and exterior multiplication
\[
\gamma_-(w) \gamma_+(v) + \gamma_+(v) \gamma_-(w) = \langle w, v \rangle 1,
\quad w \in \mathfrak{u}_-, \ v \in \mathfrak{u}_+.
\]
Since we are using dual bases we have  $\langle w_i, v_j \rangle = \delta_{i j}$. Hence
\[
D^2 \sim \sum_i E_{\xi_i} F_{\xi_j} \otimes 1.
\]
Finally using $B(E_{\xi_i}, F_{\xi_i}) = 1$, since we make the identification using dual bases, we observe that $C \sim \sum_i E_{\xi_i} F_{\xi_i}$, where $C$ is the quadratic Casimir of $\mathfrak{g}$.
Therefore we conclude that $D^{2} \sim C \otimes 1$.
We call this result the \emph{Parthasarathy formula}, although the full formula also includes the information about the terms belonging to $U(\mathfrak{l})$ that we neglected.
See \cite{partha} for the original reference and \cite{dirac-book} for a very readable textbook treatment.

\subsection{The quantum setting}

Let us now switch gears and move to the quantum setting.
We will consider the setup of \cite{qflag2}, where Dolbeault-Dirac operators on quantized irreducible flag manifolds are defined.
We will start with the definition of symmetric and exterior algebras according to \cite{bezw}.
Recall that, given two $U_q(\mathfrak{g})$-modules $V$ and $W$, there exists a \emph{braiding} $\hat{R}_{V W}: V \otimes W \to W \otimes V$, which gives an equivalence of representations.
The quantum symmetric and exterior algebras are then defined by
\[
S_q(V) = T(V) / \langle \ker(\sigma_{V V} + \mathrm{id}),\quad
\Lambda_q(V) = T(V) / \langle \ker(\sigma_{V V} - \mathrm{id}).
\]
Here $\ker(\sigma_{V V} \pm \mathrm{id})$ coincides with the span of the eigenspaces of the braiding $\hat{R}_{V V}$ with negative (respectively positive) eigenvalues.
While this definition is general, it is only for certain modules that these graded algebras have the same graded dimensions as in the classical case.
This is indeed the case for the modules in which we are interested.

Since the $U_q(\mathfrak{l})$-modules $\mathfrak{u}_\pm$ are irreducible, there there is a unique $U_q(\mathfrak{l})$-invariant pairing $\langle \cdot, \cdot \rangle : \mathfrak{u}_- \otimes \mathfrak{u}_+ \to \mathbb{C}$, up to a scalar.
It can be extended to a pairing $\langle \cdot, \cdot \rangle_k : \Lambda_q^k(\mathfrak{u}_-) \otimes \Lambda_q^k(\mathfrak{u}_+) \to \mathbb{C}$ as in \cite[Proposition 3.6]{qflag2}.
The module $\mathfrak{u}_+$ acts on $\Lambda_q(\mathfrak{u}_+)$ by left multiplication, denoted by $\gamma_+$.
We also obtain an action of $\mathfrak{u}_-$ on $\Lambda_q(\mathfrak{u}_+)$ by dualizing right multiplication on $\Lambda_q(\mathfrak{u}_-)$.
We denote this action by $\gamma_-$.
By \cite[Theorem 5.1]{qflag2} the map $\Lambda_q(\mathfrak{u}_-) \otimes \Lambda_q(\mathfrak{u}_+) \to \mathrm{End}_\mathbb{C}(\Lambda_q(\mathfrak{u}_+))$ is an equivariant isomorphism.
Hence the algebra $\mathrm{End}_\mathbb{C}(\Lambda_q(\mathfrak{u}_+))$, together with its factorization in terms of $\gamma_-$ and $\gamma_+$, can be considered a \emph{quantum Clifford algebra}.

\begin{remark}
This is the definition appearing in \cite[Definition 5.2]{qflag2}.
Nevertheless we argued in \cite{mat-proj} that, in order to recover the relations of the classical Clifford algebra, we have to choose appropriate scalars in the definition of $\langle \cdot, \cdot \rangle_k : \Lambda_q^k(\mathfrak{u}_-) \otimes \Lambda_q^k(\mathfrak{u}_+) \to \mathbb{C}$.
This small modification does not change the main results of the cited paper.
\end{remark}

Similarly to the classical case, we define the element
\[
\eth = \sum_i S^{-1}(E_{\xi_i}) \otimes \gamma_{-}(w_i) \in U_q(\mathfrak{g}) \otimes \mathrm{Cl}_q.
\]
Abstractly, $\eth$ can be seen as the Koszul differential $\sum_i v_i \otimes w_i \in S_q(\mathfrak{u}_+)^\mathrm{op} \otimes \Lambda_q(\mathfrak{u}_-)$.
This picture makes it apparent that $\eth^2 = 0$.
The differential is then embedded into $U_q(\mathfrak{g}) \otimes \mathrm{Cl}_q$ via $S^{-1} \otimes \gamma_-$. The antipode appears since we embed the opposite algebra $S_q(\mathfrak{u}_+)^\mathrm{op}$.

The \emph{Dolbeault-Dirac operator} is then defined to be $D = \eth + \eth^*$, in perfect analogy with the classical case. We will also adopt the following notation.

\begin{notation}
For $\xi_i \in \Delta(\mathfrak{u}_+)$ we will write $\mathcal{E}_i = S^{-1}(E_{\xi_i})$.
\end{notation}

\subsection{Computation of $D^2$}

We now come to the computation of $D^2$, with $D$ a Dolbeault-Dirac operator as above.
In particular, we want to investigate whether this operator takes a simple form, similar to that given by the Parthasarathy formula in the classical case.
Recall that classically we have $D^2 \sim C \otimes 1$, where $C$ is the quadratic Casimir of $\mathfrak{g}$.
In the quantum setting we do not expect to have such a simple formula, hence we will look for a weaker form.

As we discussed in the introduction of the paper, we would like to obtain an expression of the form $\sum_i C_i \otimes T_i$ for $D^2$, where $C_i \in U_q(\mathfrak{g})$ are central elements and $T_i \in \mathrm{Cl}_q$ are elements of the quantum Clifford algebra.
This would allow the computation of the spectrum of $D$ from the representation theory of $U_q(\mathfrak{g})$ and to check the compact resolvent condition.

\begin{remark}
We have shown in \cite{mat-proj} that for Dolbeault-Dirac operators on quantum projective spaces we obtain such a simple form.
The result in this case is that $D^2 \sim C \otimes T$, where $C \in U_q(\mathfrak{g})$ is a central elements and $T$ is a certain diagonal matrix.
\end{remark}

The commutation relations obtained in \autoref{thm:comm-rel} have some important consequences for the computation of $D^2$.
These are summarized in the next proposition.

\begin{proposition}
\label{prop:parthasarathy}
Let $D$ be a Dolbeault-Dirac operator as above.
\begin{enumerate}
\item We have $D^2 \sim \sum_{i, j} \mathcal{E}_i \mathcal{E}_j^* \otimes T_{i j}$, where the operators $T_{i j} \in \mathrm{Cl}_q$ are given by
\[
T_{i j} = \gamma_{-}(w_i) \gamma_{-}(w_j)^* + \sum_{k, l} b_{k, l}^{i, j} \gamma_{-}(w_k)^* \gamma_{-}(w_l),
\quad b_{k, l}^{i, j} \in \mathbb{C}.
\]
\item Suppose that $D$ satisfies the relation $D^2 \sim \sum_i C_i \otimes T_i$, where $C_i \in U_q(\mathfrak{g})$ are central elements and $T_i \in \mathrm{Cl}_q$. Then we must have $T_{i j} = 0$ for $i \neq j$.
\end{enumerate}
\end{proposition}

\begin{proof}
(1) In \cite[Proposition 5.5]{qflag2} it is proven that $\eth^2 = 0$. Therefore
\[
D^2 = \sum_{i, j} \mathcal{E}_i \mathcal{E}_j^* \otimes \gamma_{-}(w_i) \gamma_{-}(w_j)^* + \sum_{i, j} \mathcal{E}_j^* \mathcal{E}_i \otimes \gamma_{-}(w_j)^* \gamma_{-}(w_i).
\]
We rewrite this expression using the commutation relations given in \autoref{cor:comm-rel}, keeping in mind that $\mathcal{E}_i = S^{-1}(E_{\xi_i})$.
Then we obtain
\[
\begin{split}
D^2
& \sim \sum_{i, j} \mathcal{E}_i \mathcal{E}_j^* \otimes (\gamma_{-}(w_i) \gamma_{-}(w_j)^* + q^{-(\xi_i, \xi_j)} \gamma_{-}(w_j)^* \gamma_{-}(w_i)) \\
& + \sum_{i, j, k, l} c_{i, j}^{k, l} \mathcal{E}_k \mathcal{E}_l^* \otimes \gamma_{-}(w_j)^* \gamma_{-}(w_i).
\end{split}
\]
Upon relabeling the sum in the second term we can rewrite this as
\[
D^2 \sim \sum_{i, j} \mathcal{E}_i \mathcal{E}_j^* \otimes \left( \gamma_{-}(w_i) \gamma_{-}(w_j)^* + \sum_{k, l} b_{k, l}^{i, j} \gamma_{-}(w_k)^* \gamma_{-}(w_l) \right).
\]
Notice that the term $q^{-(\xi_i, \xi_j)} \gamma_{-}(w_j)^* \gamma_{-}(w_i)$ is included in the sum.
Then defining the term in parentheses as $T_{i j}$ we obtain the result.

(2) Suppose that we have the relation $D^2 \sim \sum_i C_i \otimes T_i$, where $C_i \in U_q(\mathfrak{g})$ are central elements and $T_i \in \mathrm{Cl}_q$.
Comparing with the expression $D^2 \sim \sum_{i, j} \mathcal{E}_i \mathcal{E}_j^* \otimes T_{i j}$ obtained above, we conclude that the elements $C_i$ are linear combinations of the elements $\mathcal{E}_i \mathcal{E}_j^*$.

Now recall that we have the relations $K_k E_\xi = q^{(\alpha_k, \xi)} E_\xi K_k$ and $K_k E_\xi^* = q^{-(\alpha_k, \xi)} E_\xi^* K_k$, with similar relations holding for $\mathcal{E}_i$ and $\mathcal{E}_i^*$.
Then the element $\mathcal{E}_i \mathcal{E}_j^*$ does not commute with the generators $\{K_k\}_k$ unless $i = j$.
It is easy to see, using the PBW theorem for $U_q(\mathfrak{g})$, that the vectors $\{ \mathcal{E}_i \mathcal{E}_j^* \}_{i, j}$ are linearly independent.
Therefore such terms can not appear if we assume the relation $D^2 \sim \sum_i C_i \otimes T_i$.
This is possible if and only if $T_{i j} = 0$ for $i \neq j$.
\end{proof}

The upshot is that, in order to have a result of the form $D^2 \sim \sum_i C_i \otimes T_i$, we need to have some quadratic relations in the quantum Clifford algebra.
This is because we should have $T_{i j} = 0$ for $i \neq j$ and $T_{i j}$ are quadratic expressions of $\gamma_-(w_i)$ and their adjoints.

The rest of the paper will be devoted to prove the following claim: we can find a quantum Clifford algebra coming from an irreducible flag manifold such that we do not have these quadratic relations.
Hence we do not get an analogue of the Parthasarathy formula for the class of quantized irreducible flag manifolds, as one might have hoped for.

\section{Braiding and exterior algebras}
\label{sec:braiding}

From this section on we will focus on the case of the Lagrangian Grassmannian $LG(2,4)$.
The aim of this section is to determine explicitely the braiding corresponding to the nilradical $\mathfrak{u}_+$, which corresponds to the adjoint module of $U_q(\mathfrak{sl}(2))$. This in turn will give the relations for the exterior algebras $\Lambda_q(\mathfrak{u}_+)$ and $\Lambda_q(\mathfrak{u}_-)$. We will also compute their pairing.

\subsection{Lagrangian Grassmannian}

The Lagrangian Grassmannian $LG(2,4)$ is the irreducible flag manifold obtained from $C_2 = \mathfrak{sp}(4)$ by removing the simple root $\alpha_2$, see \autoref{tab:class-comin}. In our conventions $\alpha_1$ is the short root and $\alpha_2$ is the long root. The positive roots are
\[
\alpha_1, \ \alpha_2, \ \alpha_1
+ \alpha_2, \ 2 \alpha_1 + \alpha_2.
\]
Removing the long root $\alpha_2$ corresponds to $S = \Pi \backslash \{ \alpha_2 \} = \{ \alpha_1 \}$. Then we have
\[
\Delta(\mathfrak{l}) = \{ \pm \alpha_1 \}, \quad
\Delta(\mathfrak{u}_+) =  \{ \alpha_2, \ \alpha_1
+ \alpha_2, \ 2 \alpha_1 + \alpha_2 \}.
\]
By definition we have $\mathfrak{l} = \mathfrak{g}_{-\alpha_1} \oplus \mathfrak{h} \oplus \mathfrak{g}_{\alpha_1}$, from which we obtain the isomorphism $\mathfrak{l} \cong \mathfrak{gl}(2)$ with semisimple part $\mathfrak{sl}(2)$. We have that $\mathfrak{u}_+$ is a $3$-dimensional simple $\mathfrak{sl}(2)$-module, hence it can be identified with the adjoint representation.

\begin{remark}
There is an isomorphism of the quadric $\mathbb{Q}^3$ with $LG(2,4)$.
Indeed we have the low-dimensional isomorphism $B_2 \cong C_2$ and we are removing the long root from both Dynkin diagrams.
Hence it would be equivalent for our purposes to consider this case.
\end{remark}

\subsection{The braiding}

We will now determine explicitely the braiding $\hat{R}$ corresponding to the adjoint representation of $U_{q}(\mathfrak{sl}(2))$, corresponding to the module $\mathfrak{u}_+$.
We denote by $\alpha$ the simple root of $\mathfrak{sl}(2)$ and
by $\{K,E,F\}$ the generators of $U_{q}(\mathfrak{sl}(2))$.

\begin{notation}
\label{not:adjoint-action}
Denote by $\mathfrak{u}_+$ the vector space spanned by $v_{1}$, $v_{0}$ and $v_{-1}$.
We realize the adjoint representation of $U_{q}(\mathfrak{sl}(2))$ on this vector space by the formulae
\[
\begin{gathered}
K v_{1} = q^{2} v_{1}, \quad K v_{0} = v_{0}, \quad K v_{-1} = q^{-2} v_{-1}, \\
E v_{1} = 0, \quad E v_{0} = [2]^{1/2} v_{1}, \quad E v_{-1} = [2]^{1/2} v_{0}, \\
F v_{1} = [2]^{1/2} v_{0}, \quad F v_{0} = [2]^{1/2} v_{-1}, \quad F v_{-1} = 0.
\end{gathered}
\]
With these conventions $v_{1}$ is the highest weight vector and $v_{-1}$ is the lowest weight vector.
\end{notation}

Recall that the braiding $\hat{R}_{V, W} : V \otimes W \to W \otimes V$ is uniquely determined by the relation
\[
\hat{R}_{V,W} (v \otimes w) = q^{(\mathrm{wt}(v), \mathrm{wt}(w))} w \otimes v + \sum_i w_i \otimes v_i,
\]
where $\mathrm{wt}(w_i) > \mathrm{wt}(w)$ and $\mathrm{wt}(v_i) < \mathrm{wt}(v)$.
Concretely we can start from the highest weight vectors and obtain the other values using the action of the quantized enveloping algebra.

\begin{proposition}
The braiding $\hat{R}: \mathfrak{u}_+ \otimes \mathfrak{u}_+ \to \mathfrak{u}_+ \otimes \mathfrak{u}_+$, corresponding to the adjoint representation of $U_q(\mathfrak{sl}(2))$, is given by the formulae
\[
\begin{gathered}
\hat{R}(v_1 \otimes v_1) = q^2 v_1 \otimes v_1, \quad
\hat{R}(v_1 \otimes v_0) = v_0 \otimes v_1 + (q^2 - q^{-2}) v_1 \otimes v_0, \\
\hat{R}(v_1 \otimes v_{-1}) = q^{-2} v_{-1} \otimes v_1 + q^{-1} (q - q^{-1}) (q^2 - q^{-2}) v_1 \otimes v_{-1} + (q^2 - q^{-2}) v_0 \otimes v_0, \\
\hat{R}(v_0 \otimes v_1) = v_1 \otimes v_0, \quad
\hat{R}(v_0 \otimes v_0) = v_0 \otimes v_0 + q^{-2} (q^2 - q^{-2}) v_1 \otimes v_{-1}, \\
\hat{R}(v_0 \otimes v_{-1}) = v_{-1} \otimes v_0 + (q^2 - q^{-2}) v_0 \otimes v_{-1}, \quad
\hat{R}(v_{-1} \otimes v_1) = q^{-2} v_1 \otimes v_{-1}, \\
\hat{R}(v_{-1} \otimes v_0) = v_0 \otimes v_{-1}, \quad
\hat{R}(v_{-1} \otimes v_{-1}) = q^2 v_{-1} \otimes v_{-1}.
\end{gathered}
\]
\end{proposition}

\begin{proof}
We start with some general considerations.
Let $w_{\mathrm{hw}}$ be the highest weight vector of $W$. Then
there is no $w_i$ such that $\mathrm{wt}(w_i) > \mathrm{wt}(w_{\mathrm{hw}})$.
Hence we obtain
\[
\hat{R}_{V,W} (v \otimes w_{\mathrm{hw}}) = q^{(\mathrm{wt}(v), \mathrm{wt} (w_{\mathrm{hw}}))} w_{\mathrm{hw}} \otimes v.
\]
Similarly, let $v_{\mathrm{lw}}$ be the lowest weight vector of $V$.
Then there is no $v_i$ such that $\mathrm{wt}(v_i) < \mathrm{wt}(v_{\mathrm{lw}})$.
Therefore we conclude that
\[
\hat{R}_{V,W} (v_{\mathrm{lw}} \otimes w) = q^{(\mathrm{wt}(v_{\mathrm{lw}}), \mathrm{wt}(w))} w\otimes v_{\mathrm{lw}}.
\]
We apply the arguments above to the case of the adjoint representation.
In our notations $v_{1}$ is the highest weight vector and $v_{-1}$
is the lowest weight vector. We obtain immediately
\[
\begin{gathered}\hat{R}(v_{1}\otimes v_{1})=q^{2}v_{1}\otimes v_{1},\quad\hat{R}(v_{0}\otimes v_{1})=v_{1}\otimes v_{0},\quad\hat{R}(v_{-1}\otimes v_{1})=q^{-2}v_{1}\otimes v_{-1},\\
\hat{R}(v_{-1}\otimes v_{0})=v_{0}\otimes v_{-1},\quad\hat{R}(v_{-1}\otimes v_{-1})=q^{2}v_{-1}\otimes v_{-1}.
\end{gathered}
\]
Therefore it only remains to determine the values of
\[
\hat{R}(v_{1}\otimes v_{0}),\quad\hat{R}(v_{1}\otimes v_{-1}),\quad\hat{R}(v_{0}\otimes v_{0}),\quad\hat{R}(v_{0}\otimes v_{-1}).
\]
To determine these we can proceed as follows. Since $\hat{R}$ is
a module map, we have in particular that $F\hat{R}(v\otimes w)=\hat{R}F(v\otimes w)$.
Computing the LHS and RHS separately, we obtain equations determining
the missing elements.

Let us start by computing the action of $F$ on the tensor product.
Recall that $F(v\otimes w)=Fv\otimes K^{-1}w+v\otimes Fw$. Then on
the basis elements we get
\[
\begin{gathered}F(v_{1}\otimes v_{1})=[2]^{1/2}(q^{-2}v_{0}\otimes v_{1}+v_{1}\otimes v_{0}),\quad F(v_{1}\otimes v_{0})=[2]^{1/2}(v_{0}\otimes v_{0}+v_{1}\otimes v_{-1}),\\
F(v_{1}\otimes v_{-1})=[2]^{1/2}q^{2}v_{0}\otimes v_{-1},\quad F(v_{0}\otimes v_{1})=[2]^{1/2}(q^{-2}v_{-1}\otimes v_{1}+v_{0}\otimes v_{0}),\\
F(v_{0}\otimes v_{0})=[2]^{1/2}(v_{-1}\otimes v_{0}+v_{0}\otimes v_{-1}),\quad F(v_{0}\otimes v_{-1})=[2]^{1/2}q^{2}v_{-1}\otimes v_{-1},\\
F(v_{-1}\otimes v_{1})=[2]^{1/2}v_{-1}\otimes v_{0},\quad F(v_{-1}\otimes v_{0})=[2]^{1/2}v_{-1}\otimes v_{-1},\quad F(v_{-1}\otimes v_{-1})=0.
\end{gathered}
\]

\textbf{Case $\hat{R}(v_{1}\otimes v_{0})$}. We compute
\[
\begin{split}\hat{R}F(v_{1}\otimes v_{1}) & =[2]^{1/2}(q^{-2}v_{1}\otimes v_{0}+\hat{R}(v_{1}\otimes v_{0})),\\
F\hat{R}(v_{1}\otimes v_{1}) & =[2]^{1/2}(v_{0}\otimes v_{1}+q^{2}v_{1}\otimes v_{0}).
\end{split}
\]
From these we conclude that
\[
\hat{R}(v_{1}\otimes v_{0})=v_{0}\otimes v_{1}+(q^{2}-q^{-2})v_{1}\otimes v_{0}.
\]

\textbf{Case $\hat{R}(v_{0}\otimes v_{0})$}. We compute
\[
\begin{split}\hat{R}F(v_{0}\otimes v_{1}) & =[2]^{1/2}(q^{-4}v_{1}\otimes v_{-1}+\hat{R}(v_{0}\otimes v_{0})),\\
F\hat{R}(v_{0}\otimes v_{1}) & =[2]^{1/2}(v_{0}\otimes v_{0}+v_{1}\otimes v_{-1}).
\end{split}
\]
From these we conclude that
\[
\hat{R}(v_{0}\otimes v_{0})=v_{0}\otimes v_{0}+q^{-2}(q^{2}-q^{-2})v_{1}\otimes v_{-1}.
\]

\textbf{Case $\hat{R}(v_{0}\otimes v_{-1})$}. We compute
\[
\hat{R}F(v_{0}\otimes v_{0})=[2]^{1/2}(v_{0}\otimes v_{-1}+\hat{R}(v_{0}\otimes v_{-1})).
\]
For the other term instead we get 
\[
\begin{split}F\hat{R}(v_{0}\otimes v_{0}) & =F(v_{0}\otimes v_{0})+q^{-2}(q^{2}-q^{-2})F(v_{1}\otimes v_{-1})\\
 & =[2]^{1/2}(v_{-1}\otimes v_{0}+(q^{2}+1-q^{-2})v_{0}\otimes v_{-1}).
\end{split}
\]
Therefore we obtain
\[
\hat{R}(v_{0}\otimes v_{-1})=v_{-1}\otimes v_{0}+(q^{2}-q^{-2})v_{0}\otimes v_{-1}.
\]

\textbf{Case $\hat{R}(v_{1}\otimes v_{-1})$}. This is the most complicated
case. First we compute
\[
\begin{split}\hat{R}F(v_{1}\otimes v_{0}) & =[2]^{1/2}(\hat{R}(v_{0}\otimes v_{0})+\hat{R}(v_{1}\otimes v_{-1}))\\
 & =[2]^{1/2}(v_{0}\otimes v_{0}+q^{-2}(q^{2}-q^{-2})v_{1}\otimes v_{-1}+\hat{R}(v_{1}\otimes v_{-1})).
\end{split}
\]
On the other hand we have
\[
\begin{split}F\hat{R}(v_{1}\otimes v_{0}) & =F(v_{0}\otimes v_{1})+(q^{2}-q^{-2})F(v_{1}\otimes v_{0})\\
 & =[2]^{1/2}((q^{2}-q^{-2})v_{1}\otimes v_{-1}+(q^{2}+1-q^{-2})v_{0}\otimes v_{0}+q^{-2}v_{-1}\otimes v_{1}).
\end{split}
\]
From these equations we conclude that
\[
\hat{R}(v_{1}\otimes v_{-1})=q^{-2}v_{-1}\otimes v_{1}+q^{-1}(q-q^{-1})(q^{2}-q^{-2})v_{1}\otimes v_{-1}+(q^{2}-q^{-2})v_{0}\otimes v_{0}.
\qedhere
\]
\end{proof}

\begin{remark}
It is straightforward, although tedious, to verify directly in terms of the formulae given above that the braiding $\hat{R}$ satisfies the Yang-Baxter equation
\[
\hat{R}_{12}\hat{R}_{23}\hat{R}_{12}=\hat{R}_{23}\hat{R}_{12}\hat{R}_{23},
\]
where we use the standard leg-numbering notation.
\end{remark}

\subsection{The algebra $\Lambda_q(\mathfrak{u}_+)$}

We will now use the braiding $\hat{R}$ to determine the relations of the quantum exterior algebra corresponding to $\mathfrak{u}_+$.
As we have recalled in a previous section, by definition the quantum exterior algebra $\Lambda_q(V)$ is the quotient of the tensor algebra $T(V)$ by the two-sided ideal generated by quantum symmetric $2$-tensors.
These are defined in terms of the braiding $\hat{R}_V$ corresponding to $V$. We denote by $S_q^2 V$ (respectively $\Lambda_q^2 V$)
the span of the eigenvectors of $\hat{R}_V$ with positive (respectively
negative) eigenvalues.

Therefore to proceed we will need the eigenvalues and eigenvectors of the braiding $\hat{R}$.

\begin{lemma}
The eigenvectors and eigenvalues of $\hat{R}$ are given by
\[
\begin{gathered}
\{v_{1} \otimes v_{1}, \ q^{2}\}, \quad
\{v_{1} \otimes v_{0} + q^{-2} v_{0} \otimes v_{1}, \ q^{2}\}, \quad
\{v_{1} \otimes v_{0} - q^{2} v_{0} \otimes v_{1}, \ -q^{-2}\}, \\
\{v_{1} \otimes v_{-1} + q^{-4} v_{-1} \otimes v_{1} + q^{-1}(q + q^{-1}) v_{0} \otimes v_{0}, \ q^{2}\}, \\
\{v_{1} \otimes v_{-1} - v_{-1} \otimes v_{1} - q(q - q^{-1}) v_{0} \otimes v_{0}, \ -q^{-2}\}, \\
\{v_{1} \otimes v_{-1} + q^{2} v_{-1}\otimes v_{1} - q^{2} v_{0} \otimes v_{0}, \ q^{-4}\}, \\
\{v_{0} \otimes v_{-1} + q^{-2} v_{-1} \otimes v_{0}, \ q^{2}\}, \quad
\{v_{0} \otimes v_{-1} - q^{2} v_{-1} \otimes v_{0},\ -q^{-2}\}, \quad
\{v_{-1} \otimes v_{-1}, \ q^{2}\}.
\end{gathered}
\]
\end{lemma}

\begin{proof}
Follows from simple computations that we omit.
\end{proof}

\begin{remark}
We have three different eigenvalues, namely $q^2$, $-q^{-2}$ and $q^{-4}$.
Hence $\hat{R}$ does not satisfy a quadratic relation, differently from the fundamental representation of $U_q(\mathfrak{sl}(2))$.
\end{remark}

In particular the spaces of symmetric and antisymmetric $2$-tensors are given by
\[
\begin{split}
S_q^2 \mathfrak{u}_+ =\mathrm{span} \{
& v_1 \otimes v_1, \
v_1 \otimes v_0 + q^{-2} v_0 \otimes v_1, \
v_0 \otimes v_{-1} + q^{-2} v_{-1}\otimes v_0, \\
& v_{-1} \otimes v_{-1}, \
v_1 \otimes v_{-1} + q^2 v_{-1} \otimes v_1 - q^2 v_0 \otimes v_0, \\
& v_1 \otimes v_{-1} + q^{-4} v_{-1} \otimes v_1 + q^{-1} (q + q^{-1}) v_0 \otimes v_0\}, \\
\Lambda_q^2 \mathfrak{u}_+ =\mathrm{span}\{
& v_1 \otimes v_0 - q^2 v_0 \otimes v_1, \
v_0 \otimes v_{-1} - q^2 v_{-1} \otimes v_0, \\
& v_1 \otimes v_{-1} - v_{-1} \otimes v_1 - q (q - q^{-1}) v_0 \otimes v_0 \}.
\end{split}
\]

We are now ready to derive the relations for the exterior algebra.

\begin{proposition}
\label{prop:relations-UP}
The algebra $\Lambda_q (\mathfrak{u}_+)$ has the relations
\[
\begin{gathered}
v_1 \wedge v_1 = 0, \quad
v_0 \wedge v_1 = - q^2 v_1 \wedge v_0, \quad
v_0 \wedge v_0 = - q^{-1} (q - q^{-1}) v_1 \wedge v_{-1}, \\
v_{-1} \wedge v_1 = - v_1 \wedge v_{-1}, \quad
v_{-1} \wedge v_0 = -q^2 v_0 \wedge v_{-1}, \quad
v_{-1} \wedge v_{-1} = 0.
\end{gathered}
\]
\end{proposition}

\begin{proof}
By definition, the algebra $\Lambda_{q}(\mathfrak{u}_{+})$ is the
quotient of $T(\mathfrak{u}_{+})$ by the ideal generated by the subspace
of eigenvectors of $\hat{R}$ with positive eigenvalues, that is $S_{q}^{2}\mathfrak{u}_{+}$.
The relations follow straightforwardly from the description of $S_{q}^{2}\mathfrak{u}_{+}$,
except for two of them which we describe below. Taking the quotient gives the relations
\[
v_{1}\wedge v_{-1}+q^{2}v_{-1}\wedge v_{1}-q^{2}v_{0}\wedge v_{0}=0,\quad v_{1}\wedge v_{-1}+q^{-4}v_{-1}\wedge v_{1}+q^{-1}(q+q^{-1})v_{0}\wedge v_{0}=0.
\]
Upon taking appropriate linear combinations, we see that these are
equivalent to
\[
v_{1}\wedge v_{-1}=-v_{-1}\wedge v_{1},\quad v_{0}\wedge v_{0}=-q^{-1}(q-q^{-1})v_{1}\wedge v_{-1}.
\qedhere
\]
\end{proof}

\subsection{The algebra $\Lambda_q(\mathfrak{u}_-)$}

In the classical case, the $\mathfrak{l}$-module $\mathfrak{u}_{-}$ can be identified with the dual of $\mathfrak{u}_{+}$ with respect to the invariant Killing form.
In the special case of $\mathfrak{u}_{+}$ being the adjoint module of $\mathfrak{sl}(2)$, we also have the identification $\mathfrak{u}_+ \cong\mathfrak{u}_-$.
This also holds for the corresponding $U_q(\mathfrak{sl}(2))$-modules.
We will derive an explicit formula below.

We denote by $\{w_i\}_i$ the basis of $\mathfrak{u}_-$ dual to the basis $\{v_i\}_i$ of $\mathfrak{u}_+$, that is $\langle w_i, v_j \rangle = \delta_{i j}$.
Here the dual pairing $\langle \cdot, \cdot \rangle : \mathfrak{u}_- \to \mathfrak{u}_+$ is assumed to be invariant under to the action of $U_q(\mathfrak{l})$.
This means that for any $w \in \mathfrak{u}_-$, $v \in \mathfrak{u}_+$ and $X \in U_q(\mathfrak{l})$ we should have
\[
\langle X_{(1)} w, X_{(2)} v \rangle = \varepsilon(X) \langle w, v \rangle.
\]
In this setting the dual pairing is unique up to a scalar factor.

\begin{lemma}
\label{lem:iso-modules}
We have an isomorphism $\psi:\mathfrak{u}_+ \to \mathfrak{u}_-$ of $U_q(\mathfrak{sl}(2))$-modules given by
\[
\psi(v_1) = w_{-1}, \quad
\psi(v_0) = -w_{0}, \quad
\psi(v_{-1}) = q^2 w_1.
\]
\end{lemma}

\begin{proof}
Clearly we have an isomorphism of vector spaces, so we only have to
check equivariance. First of all, the condition $\langle K w, K v \rangle = \langle w, v \rangle$ gives
\[
K w_{-1} = q^2 w_{-1}, \quad
K w_0 = w_0, \quad
K w_1 = q^{-2} w_1.
\]
It is clear then that an equivariant map $\psi: \mathfrak{u}_+ \to \mathfrak{u}_-$ should have the form
\[
\psi(v_1) = \alpha w_{-1}, \quad
\psi(v_0) = \beta w_0, \quad
\psi(v_{-1}) = \gamma w_1.
\]

Now we want to determine the action of $F$ on the dual basis.
Invariance of the dual pairing gives the condition $\langle F w, K^{-1} v\rangle = - \langle w, F v \rangle$ for any $w \in \mathfrak{u}_-$ and $v \in \mathfrak{u}_+$.
We will use the conventions of \autoref{not:adjoint-action}.
For $w = w_0$ and $v = v_1$ we get $\langle F w_{0}, q^{-2} v_1 \rangle = - \langle w_0, [2]^{1/2} v_{0}\rangle$. This implies $F w_0 = - [2]^{1/2} q^2 w_1$. Similarly, for $w = w_{-1}$ and $v = v_0$ we get $\langle Fw_{-1}, v_0 \rangle = - \langle w_{-1}, [2]^{1/2} v_{-1}\rangle$,
which implies $F w_{-1} = - [2]^{1/2} w_{0}$. Finally it is clear that
$F w_1 = 0$. Summarizing, the action of $F$ on the basis $\{w_i\}_i$ is given by
\[
F w_{-1} = - [2]^{1/2} w_0, \quad
F w_0 = - [2]^{1/2} q^2 w_1, \quad
F w_1 = 0.
\]
Similarly one can obtain the action of $E$.
Using these formulae we compute
\[
F \psi(v_1) = - [2]^{1/2} \alpha w_0, \quad
F \psi(v_0) = - [2]^{1/2} q^2 \beta w_1, \quad
F \psi(v_{-1}) = 0.
\]
On the other hand we have
\[
\psi(F v_1) = [2]^{1/2} \beta w_0, \quad
\psi(F v_0) = [2]^{1/2} \gamma w_1, \quad
\psi(F v_{-1}) = 0.
\]
Enforcing equivariance of $\psi$, namely the condition $F \psi = \psi F$, we find the relations $\alpha = -\beta = q^{-2} \gamma$.
One can show that the action of $E$ gives the same conditions.
Finally we can fix the value $\alpha = 1$ to obtain the expression in the claim.
\end{proof}

\begin{remark}
This isomorphism can be extended to an isomorphism of tensor powers of these modules. For example we obtain $\psi: \mathfrak{u}_+ \otimes \mathfrak{u}_+ \to \mathfrak{u}_- \otimes \mathfrak{u}_-$ by setting $\psi(w \otimes w^\prime) = \psi(w) \otimes \psi(w^\prime)$.
The equivariance of this map follows from that of $\psi:\mathfrak{u}_+ \to \mathfrak{u}_-$.
\end{remark}

It is now immediate to obtain the relations for $\Lambda_{q}(\mathfrak{u}_{-})$.

\begin{corollary}
\label{cor:relations-UM}
The algebra $\Lambda_{q}(\mathfrak{u}_{-})$ has the relations
\[
\begin{gathered}w_{-1}\wedge w_{-1}=0,\quad w_{0}\wedge w_{-1}=-q^{2}w_{-1}\wedge w_{0},\quad w_{0}\wedge w_{0}=-q(q-q^{-1})w_{-1}\wedge w_{1},\\
w_{1}\wedge w_{-1}=-w_{-1}\wedge w_{1},\quad w_{1}\wedge w_{0}=-q^{2}w_{0}\wedge w_{1},\quad w_{1}\wedge w_{1}=0.
\end{gathered}
\]
\end{corollary}

\begin{proof}
Follows from the relations of $\Lambda_{q}(\mathfrak{u}_{+})$ and the isomorphism $\psi$.
\end{proof}

\subsection{Computation of pairings}

We are now in the position to compute the pairing between $\Lambda_{q}(\mathfrak{u}_-)$ and $\Lambda_{q}(\mathfrak{u}_+)$.
Let us briefly recall how this is defined.
First of all there is a natural pairing of the tensor algebras $T(\mathfrak{u}_-)$ and $T(\mathfrak{u}_+)$, which extends the dual pairing of $\mathfrak{u}_-$ and $\mathfrak{u}_+$.
By definition the exterior algebra $\Lambda_q(V)$ is a quotient of the tensor algebra $T(V)$.
It is shown in \cite[Proposition 3.2]{ChTS14} that the map $\pi : \Lambda^k_q V  \to \Lambda^k_q (V)$, obtained by composing the inclusion $\Lambda_q^k V \hookrightarrow V^{\otimes k}$ with the quotient $V^{\otimes k} \twoheadrightarrow \Lambda_q^k (V)$, is an equivariant isomorphism.
Hence we can define a pairing of exterior algebras by
\[
\langle w, v \rangle_{\Lambda} = \langle \pi_{-}^{-1}(w), \pi_{+}^{-1}(v) \rangle_{T}, \quad
w \in \Lambda^k_q (\mathfrak{u}_-), \ v \in \Lambda^k_q (\mathfrak{u}_+).
\]
Here the subscripts $\Lambda$ and $T$ refer to the exterior and the tensor algebra, respectively.

To compute the pairing we need explicit expressions for the elements $\pi_{-}^{-1}(w)$ and $\pi_{+}^{-1}(v)$.
In degrees $0$ and $3$ we have $1$-dimensional vector spaces, hence the pairing is just a non-zero number that we are always free to rescale. In degree $1$ the pairing is simply the dual pairing. Therefore we only have to compute the pairing in degree $2$.

\begin{notation}
\label{not:bases-degree2}
We define a basis $\{ V_1, \ V_0, \ V_{-1} \}$ of $\Lambda^2_q \mathfrak{u}_+$ by
\begin{gather*}
V_1 = v_1 \otimes v_0 - q^2 v_0 \otimes v_1, \quad
V_{-1} = v_0 \otimes v_{-1} - q^2 v_{-1} \otimes v_0, \\
V_0 = v_1 \otimes v_{-1} - v_{-1} \otimes v_1 - q (q - q^{-1}) v_0 \otimes v_0.
\end{gather*}
Similarly we define a basis $\{ W_{-1}, \ W_0, \ W_1 \}$ of $\Lambda^2_q \mathfrak{u}_-$ by
\begin{gather*}
W_{-1} = w_{-1} \otimes w_0 - q^2 w_0 \otimes w_{-1}, \quad
W_1 = w_0 \otimes w_1 - q^2 w_1 \otimes w_0, \\
W_0 = w_{-1} \otimes w_1 - w_1 \otimes w_{-1} - q^{-1} (q - q^{-1}) w_0 \otimes w_0.
\end{gather*}
\end{notation}

Notice that the vectors $W_{-1}$, $W_0$ and $W_1$ correspond, up to scalars, to the image of the vectors $V_1$, $V_0$ and $V_{-1}$ under the isomorphism $\psi$ from \autoref{lem:iso-modules}.

\begin{lemma}
\label{lem:lift-tensor}
We have the identities
\begin{gather*}
v_1 \wedge v_0 = \frac{q^{-2}}{[2]_{q^2}} \pi_+(V_1), \quad
v_1 \wedge v_{-1} = \frac{1}{[2]_{q^2}} \pi_+(V_0), \quad
v_0 \wedge v_{-1} = \frac{q^{-2}}{[2]_{q^2}} \pi_+(V_{-1}), \\
w_{-1} \wedge w_0 = \frac{q^{-2}}{[2]_{q^2}} \pi_-(W_{-1}), \quad
w_{-1} \wedge w_1 = \frac{1}{[2]_{q^2}} \pi_-(W_0), \quad
w_0 \wedge w_1 = \frac{q^{-2}}{[2]_{q^2}} \pi_-(W_1).
\end{gather*}
Here we use the notation $[2]_{q^2} = q^2 + q^{-2}$.
\end{lemma}

\begin{proof}
This follows from the definition of the maps $\pi_{\pm}$, together with simple computations involving the commutation relations obtained in \autoref{prop:relations-UP} and \autoref{cor:relations-UM}. As an example, we show the identity for $v_1 \wedge v_{-1}$. We compute
\[
\begin{split}
\pi_+(V_0)
 & = v_1 \wedge v_{-1} - v_{-1} \wedge v_1 - q (q - q^{-1}) v_0 \wedge v_0 \\
 & = 2 v_1 \wedge v_{-1} + (q - q^{-1})^2 v_1 \wedge v_{-1} = (q^2 + q^{-2}) v_1 \wedge v_{-1}.
\end{split}
\]
The other relations are proven similarly.
\end{proof}

We are now in the position to compute the pairing between elements of degree $2$.

\begin{proposition}
\label{prop:pairings}
The non-zero pairings between the basis $\{w_0 \wedge w_1, \ w_{-1} \wedge w_1, \ w_{-1} \wedge w_0\}$ of $\Lambda^2_q(\mathfrak{u}_-)$ and the basis $\{v_1 \wedge v_0, \ v_1 \wedge v_{-1}, \ v_0 \wedge v_{-1}\}$ of $\Lambda^2_q(\mathfrak{u}_+)$ are given by
\[
\langle w_0 \wedge w_1, v_1 \wedge v_0 \rangle = \langle w_{-1} \wedge w_0, v_0 \wedge v_{-1} \rangle = \frac{q^{-2}}{[2]_{q^2}}, \quad
\langle w_{-1} \wedge w_1, v_1 \wedge v_{-1} \rangle = \frac{1}{[2]_{q^2}},
\]
where we use the notation $[2]_{q^2} = q^2 + q^{-2}$.\end{proposition}

\begin{proof}
The pairing of exterior algebras is defined by $\langle w, v\rangle_{\Lambda} = \langle \pi_-^{-1}(w), \pi_+^{-1}(v) \rangle_{T}$, where the pairing of tensor algebras is given by $\langle w^\prime \otimes w, v \otimes v^\prime \rangle_T = \langle w, v \rangle \langle w^\prime, v^\prime \rangle$.
Then the result follows from explicit computations using \autoref{lem:lift-tensor}. Let us see one example. We have
\[
\begin{split}
\langle w_0 \wedge w_1, v_1 \wedge v_0 \rangle
 & = \frac{q^{-4}}{[2]^2_{q^2}} \langle w_0 \otimes w_{1} - q^2 w_1 \otimes w_0, v_1 \otimes v_0 - q^2 v_0 \otimes v_1 \rangle \\
 & = \frac{q^{-4}}{[2]^2_{q^2}} (\langle w_0 \otimes w_1, v_1 \otimes v_0 \rangle + q^4 \langle w_1 \otimes w_0, v_0 \otimes v_1 \rangle) \\
 & = \frac{q^{-4}}{[2]^2_{q^2}} (1 + q^4) = \frac{q^{-2}}{[2]_{q^2}}.
\end{split}
\]
The other cases are treated similarly.
\end{proof}

As mentioned previously, we are free to rescale the pairing in each degree.

\section{The quantum Clifford algebra}

\label{sec:quantum-clifford}

In this section we will derive explicit expression for the operators $\gamma_-(w)$ on $\Lambda_{q}(\mathfrak{u}_{+})$.
These operators, together with $\gamma_+(v)$, generate the quantum Clifford algebra.
For our discussion of Dolbeault-Dirac operators we are actually interested in the adjoints $\gamma_-(w)^*$.
To define these we will classify all invariant Hermitian inner products on the exterior algebra $\Lambda_{q}(\mathfrak{u}_{+})$.

\subsection{Bases of exterior algebras}

Below we summarize our conventions for the bases of the exterior algebras $\Lambda_{q}(\mathfrak{u}_{+})$ and $\Lambda_{q}(\mathfrak{u}_{-})$.

\begin{notation}
\label{def:ord-basis}

We fix an ordered basis $\{v_i^{(k)}\}_i$ of $\Lambda_q (\mathfrak{u}_+)$ in each degree $k$. We choose
\[
\{1\},\quad
\{v_1, v_0, v_{-1}\}, \quad
\{v_1 \wedge v_0, \quad v_1 \wedge v_{-1}, \quad v_0 \wedge v_{-1}\},\quad
\{v_1 \wedge v_0 \wedge v_{-1}\}.
\]
Similarly we fix an ordered basis $\{w_i^{(k)}\}_i$ of $\Lambda_q (\mathfrak{u}_-)$
in each degree $k$. We choose
\[
\{1\}, \quad
\{w_{1},w_{0},w_{-1}\}, \quad
\{w_{0}\wedge w_{1},\quad w_{-1}\wedge w_{1},\quad w_{-1}\wedge w_{0}\}, \quad
\{w_{-1}\wedge w_{0}\wedge w_{1}\}.
\]
\end{notation}

Let us also summarize the results for the pairing $\langle \cdot, \cdot\rangle: \Lambda_q (\mathfrak{u}_-) \otimes \Lambda_q (\mathfrak{u}_+) \to \mathbb{C}$.
Elements of different degrees are orthogonal. It follows from \autoref{prop:pairings} that we have
\begin{gather*}
\langle1,1\rangle=1,\quad
\langle w_{i},v_{j}\rangle=\delta_{ij}, \\
\langle w_0 \wedge w_1, v_1 \wedge v_0 \rangle = \langle w_{-1} \wedge w_0, v_0 \wedge v_{-1} \rangle = q^{-2}, \quad
\langle w_{-1}\wedge w_{1},v_{1}\wedge v_{-1}\rangle=1, \\
\langle w_{-1}\wedge w_{0}\wedge w_{1},v_{1}\wedge v_{0}\wedge v_{-1}\rangle=1.
\end{gather*}
Notice that we have rescaled the pairings appearing in \autoref{prop:pairings}.
Also notice the following orthogonality property: we have $\langle w_{i}^{(k)},v_{j}^{(k)}\rangle=\delta_{ij}b_{i}$
for some numbers $b_{i}$.

\subsection{Action of $\mathfrak{u}_-$}

Let us recall the definition of the action $\gamma_{-}$ of the module $\mathfrak{u}_{-}$ on  the exterior algebra $\Lambda_{q}(\mathfrak{u}_{+})$. It is defined as the dual of the right multiplication on $\mathfrak{u}_{-}$, that is
\[
\langle w, \gamma_{-}(z) v\rangle_{k} = \langle w \wedge z, v\rangle_{k + 1}, \quad
z \in \mathfrak{u}_{-}, \ w \in \Lambda_{q}^{k }(\mathfrak{u}_{-}), \ v \in \Lambda_{q}^{k + 1}(\mathfrak{u}_{+}).
\]
Below we will compute the explicit action of $\mathfrak{u}_{-}$ on
the ordered basis of $\Lambda_{q}(\mathfrak{u}_{+})$.

\begin{proposition}
\label{prop:action-gammaM}
The action of $\mathfrak{u}_{-}$ on the basis of $\Lambda_{q}(\mathfrak{u}_{+})$ is as follows. In degree $0$ we have $\gamma_{-}(w_{i}) 1 = 0$, in degree $1$ we have  $\gamma_{-}(w_{i}) v_{j} = \delta_{ij} 1$, in degree $2$ we have
\[
\begin{gathered}
\gamma_{-}(w_{1})v_{1}\wedge v_{0}=q^{-2}v_{0},\quad\gamma_{-}(w_{1})v_{1}\wedge v_{-1}=v_{-1},\quad\gamma_{-}(w_{1})v_{0}\wedge v_{-1}=0,\\
\gamma_{-}(w_{0})v_{1}\wedge v_{0}=-v_{1},\quad\gamma_{-}(w_{0})v_{1}\wedge v_{-1}=-q(q-q^{-1})v_{0},\quad\gamma_{-}(w_{0})v_{0}\wedge v_{-1}=q^{-2}v_{-1},\\
\gamma_{-}(w_{-1})v_{1}\wedge v_{0}=0,\quad\gamma_{-}(w_{-1})v_{1}\wedge v_{-1}=-v_{1},\quad\gamma_{-}(w_{-1})v_{0}\wedge v_{-1}=-v_{0},\\
\end{gathered}
\]
and finally in degree $3$ we have
\[
\begin{gathered}
\gamma_{-}(w_{1}) v_{1} \wedge v_{0} \wedge v_{-1} = q^{2} v_{0} \wedge v_{-1}, \quad \gamma_{-}(w_{1}) v_{1} \wedge v_{0} \wedge v_{-1} = -q^{2} v_{1} \wedge v_{-1}, \\
\gamma_{-}(w_{1}) v_{1} \wedge v_{0} \wedge v_{-1} = q^{4} v_{1} \wedge v_{0}.
\end{gathered}
\]
\end{proposition}

\begin{proof}
We denote by $\{v_{i}^{(k)}\}_{i}$ and $\{w_{i}^{(k)}\}_{i}$ the
bases of $\Lambda_{q}^{k}(\mathfrak{u}_{+})$ and $\Lambda_{q}^{k}(\mathfrak{u}_{-})$
as in \autoref{def:ord-basis}. Recall that these are orthogonal in the
sense that $\langle w_{i}^{(k)},v_{j}^{(k)}\rangle=\delta_{ij}b_{i}$
for some numbers $b_{i}$.
From this fact it easily follows that we can write any vector $v \in \Lambda_{q}^{k}(\mathfrak{u}_{+})$ as
\[
v = \sum_{i} \frac{\langle w_{i}^{(k)}, v\rangle}{\langle w_{i}^{(k)}, v_{i}^{(k)}\rangle} v_{i}^{(k)}.
\]
Then, acting with $\gamma_{-}(w_{a}):\Lambda_{q}^{k}(\mathfrak{u}_{+})\to\Lambda_{q}^{k-1}(\mathfrak{u}_{+})$
on a vector $v\in\Lambda_{q}^{k}(\mathfrak{u}_{+})$, we obtain
\[
\gamma_{-}(w_{a})v=\sum_{i}\frac{\langle w_{i}^{(k-1)},\gamma_{-}(w_{a})v\rangle}{\langle w_{i}^{(k-1)},v_{i}^{(k-1)}\rangle}v_{i}^{(k-1)}=\sum_{i}\frac{\langle w_{i}^{(k-1)}\wedge w_{a},v\rangle}{\langle w_{i}^{(k-1)},v_{i}^{(k-1)}\rangle}v_{i}^{(k-1)}.
\]
From this expression we can easily compute the action of $\gamma_{-}(w_{a})$
on the basis $\{v_{i}^{(k)}\}_{i}$.

In degree $0$ it is clear that $\gamma_{-}(w_a) 1 = 0$. In degree
$1$ we find $\gamma_- (w_i) v_j = \delta_{ij} 1$, since $\langle w_i, v_j \rangle = \delta_{ij}$.
The formulae in degrees $2$ and $3$ follow from simple computations.
Let us just show an example. Let $v=v_{1}\wedge v_{0}\wedge v_{-1}\in\Lambda_{q}^{3}(\mathfrak{u}_{+})$.
Then we have
\[
\begin{split}
\gamma_{-}(w_{1})v
& = \frac{\langle w_{0} \wedge w_{1} \wedge w_{1}, v\rangle}{\langle w_{0} \wedge w_{1}, v_{1} \wedge v_{0}\rangle} v_{1} \wedge v_{0} + \frac{\langle w_{-1} \wedge w_{1} \wedge w_{1}, v\rangle}{\langle w_{-1} \wedge w_{1}, v_{1} \wedge v_{-1}\rangle} v_{1} \wedge v_{-1} \\
& + \frac{\langle w_{-1} \wedge w_{0} \wedge w_{1}, v\rangle}{\langle w_{-1} \wedge w_{0}, v_{0} \wedge v_{-1}\rangle} v_{0} \wedge v_{-1}\\
& = \frac{\langle w_{-1} \wedge w_{0} \wedge w_{1}, v_{1} \wedge v_{0} \wedge v_{-1}\rangle}{\langle w_{-1} \wedge w_{0}, v_{0} \wedge v_{-1}\rangle} v_{0} \wedge v_{-1} = q^{2} v_{0} \wedge v_{-1}.
\end{split}
\]
The other cases follow similarly, upon using the appropriate commutation relations. We remark that we have to use relations like $w_{-1} \wedge w_{0} \wedge w_{0} = 0$, even though $w_{0} \wedge w_{0} \neq 0$.
\end{proof}

\subsection{Hermitian inner products and adjoints}

The next task is to determine an appropriate Hermitian inner product on $\Lambda_q(\mathfrak{u}_+)$.
It should be invariant with respect to the adjoint action of $U_q(\mathfrak{l})$ considered as a $*$-algebra, with the $*$-structure coming from the compact real form of $U_q(\mathfrak{g})$.
More concretely this means that
\[
(v, X v^\prime) = (X^* v, v^\prime), \quad v, v^\prime \in \Lambda_q(\mathfrak{u}_+), \ X \in U_q(\mathfrak{l}).
\]
In the next proposition we will determine all the Hermitian inner products on $\Lambda_q(\mathfrak{u}_+)$ that satisfy this condition.
The only freedom we will get is a rescaling in each degree.

\begin{proposition}
Let $(\cdot, \cdot): \Lambda_q(\mathfrak{u}_+) \otimes \Lambda_q(\mathfrak{u}_+) \to \mathbb{C}$ be an invariant inner product as above.
Denote by $M^{(k)}$ the matrix of inner products in degree $k$, that is $M^{(k)}_{i j} = (v^{(k)}_i, v^{(k)}_j)$.
Then, up to a rescaling in each degree, we have
\[
M^{(0)} = (1), \quad
M^{(1)} = \left(\begin{array}{ccc}
1 & 0 & 0\\
0 & 1 & 0\\
0 & 0 & q^2
\end{array}\right), \quad
M^{(2)} = \left(\begin{array}{ccc}
1 & 0 & 0\\
0 & q^4 & 0\\
0 & 0 & q^2
\end{array}\right), \quad
M^{(3)} = (1).
\]
\end{proposition}

\begin{proof}
The statement in degrees $0$ and $3$ is obvious, since these are $1$-dimensional vector spaces.
Next we consider the case of degree $1$.
In general the fact that the inner product is invariant under the generator $K$, namely $(v, K v^\prime) = (K v, v^\prime)$, implies that vectors of different weights are orthogonal.
To proceed we will use the explicit action given in \autoref{not:adjoint-action}.
On one hand we have $(v_0, F v_1) = [2]^{1/2} (v_0, v_0)$. On the other hand we have
\[
(F^* v_0, v_1) = (E K^{-1} v_0, v_1) = [2]^{1/2} (v_1, v_1).
\]
Hence we must have $(v_0, v_0) = (v_1, v_1)$. Similarly we have $(v_{-1}, F v_0) = [2]^{1/2}(v_{-1}, v_{-1})$ and
\[
(F^* v_{-1}, v_0) = (E K^{-1} v_{-1}, v_0) = [2]^{1/2} q^2 (v_0, v_0).
\]
Then we conclude that $(v_{-1}, v_{-1}) = q^2 (v_0, v_0)$.
We obtain the result by fixing $(v_1, v_1) = 1$.

Finally we consider the case of degree $2$.
The spaces $\Lambda^1_q \mathfrak{u}_+$ and $\Lambda^2_q \mathfrak{u}_+$ are isomorphic as $U_q(\mathfrak{sl}(2))$-modules, but we have to be careful with this identification.
Recall that we have $\Lambda^2_q \mathfrak{u}_+ = \mathrm{span} \{V_1, V_0, V_{-1}\}$, where the vectors are given explicitely in \autoref{not:bases-degree2}.
The action of the generator $F$ on these elements is given by
\[
F V_1 = [2]^{1/2} V_0, \quad
F V_0 = [2]^{1/2} V_{-1}, \quad
F V_{-1} = 0.
\]
This follows from easy computations. For example we have
\[
\begin{split}
F V_1
 & = F v_1 \otimes K^{-1} v_0 + v_1 \otimes F v_0 - q^2 F v_0 \otimes K^{-1} v_1 - q^2 v_0 \otimes F v_1 \\
 & = [2]^{1/2} (v_0 \otimes v_0 + v_1 \otimes v_{-1} - v_{-1} \otimes v_1 - q^2 v_0 \otimes v_0) = [2]^{1/2} V_0.
\end{split}
\]
Now recall the relations between the elements $V_1, \ V_0, \ V_{-1}$ and the elements $v_1 \wedge v_0, \ v_1 \wedge v_{-1}, \ v_0 \wedge v_{-1}$ given by \autoref{lem:lift-tensor}.
Since the map $\pi_+$ is equivariant, we obtain
\[
F (v_1 \wedge v_0) = [2]^{1/2} q^{-2} v_1 \wedge v_{-1}, \quad
F (v_1 \wedge v_{-1}) = [2]^{1/2} q^2 v_0 \wedge v_{-1}, \quad
F (v_0 \wedge v_{-1}) = 0.
\]
Hence, making the identifications $v_1 \wedge v_0 \sim v_1$, $v_1 \wedge v_{-1} \sim q^2 v_0$ and $v_0 \wedge v_{-1} \sim v_{-1}$, we obtain the result in degree $2$ by appropriate rescaling of the result in degree $1$.
\end{proof}

Using the Hermitian inner product $(\cdot, \cdot): \Lambda_q(\mathfrak{u}_+) \otimes \Lambda_q(\mathfrak{u}_+) \to \mathbb{C}$ we can define adjoints of operators on $\Lambda_q(\mathfrak{u}_+)$.
In particular we are interested in operators of degree $-1$, such as $\gamma_-(w)$.
Given such an operator $T$, we will consider the maps
\[
T^{(k)}: \Lambda_q^k(\mathfrak{u}_+) \to \Lambda_q^{k - 1}(\mathfrak{u}_+),\quad
T^{(k)*}: \Lambda_q^{k - 1} (\mathfrak{u}_+) \to \Lambda_q^k(\mathfrak{u}_+).
\]
The first map is the restriction of $T$ to elements of degree $k$, while the second map is defined by $(T^{(k)} v, v^\prime)_{k - 1} = (v, T^{(k)*} v^{\prime})_k$ 
for all elements $v \in \Lambda_q^k(\mathfrak{u}_+)$ and $v^\prime \in \Lambda_q^{k - 1}(\mathfrak{u}_+)$.
In terms of matrices, this means that our operators take the form
\[
T = \left(
\begin{array}{cccc}
0 & T^{(1)} & 0 & 0\\
0 & 0 & T^{(2)} & 0\\
0 & 0 & 0 & T^{(3)}\\
0 & 0 & 0 & 0
\end{array}
\right),\quad
T^* = \left(
\begin{array}{cccc}
0 & 0 & 0 & 0\\
T^{(1)*} & 0 & 0 & 0\\
0 & T^{(2)*} & 0 & 0\\
0 & 0 & T^{(3)*} & 0
\end{array}
\right).
\]

\begin{lemma}
\label{lem:matrix-adjoint}
Let $M^{(k)}$ be  the matrix of inner products in degree $k$, that is the matrix with entries $M^{(k)}_{i j} = (v^{(k)}_i, v^{(k)}_j)_k$.
Let $T$ be an operator as above.
Then we have
\[
T^{(k)*} = (M^{(k)})^{-1} T^{(k)\dagger} M^{(k - 1)},
\]
where $\dagger$ denotes the usual conjugate transpose.
\end{lemma}

\begin{proof}
Follows from some elementary linear algebra.
\end{proof}

\subsection{Matrix expressions}

Here we will record the matrix expressions for the operators $\gamma_-(w_i)$ and their adjoints.
Let us adopt the following short-hand notation.

\begin{notation}
We will write $\Gamma_+ = \gamma_-(w_1)$, $\Gamma_0 = \gamma_-(w_0)$ and $\Gamma_- = \gamma_{-}(w_{-1})$.
\end{notation}

The matrices can be read off from \autoref{prop:action-gammaM}. In degree $1$ we have
\[
\Gamma_{+}^{(1)} = \left(
\begin{array}{ccc}
1 & 0 & 0
\end{array}
\right), \quad
\Gamma_{0}^{(1)} = \left(
\begin{array}{ccc}
0 & 1 & 0
\end{array}
\right), \quad
\Gamma_{-}^{(1)} = \left(
\begin{array}{ccc}
0 & 0 & 1
\end{array}
\right).
\]
In degree $2$ we have
\[
\Gamma_{+}^{(2)} = \left(
\begin{array}{ccc}
0 & 0 & 0\\
q^{-2} & 0 & 0\\
0 & 1 & 0
\end{array}
\right), \quad
\Gamma_{0}^{(2)} = \left(
\begin{array}{ccc}
-1 & 0 & 0\\
0 & -q (q - q^{-1}) & 0\\
0 & 0 & q^{-2}
\end{array}\right), \quad
\Gamma_{-}^{(2)} = \left(
\begin{array}{ccc}
0 & -1 & 0\\
0 & 0 & -1\\
0 & 0 & 0
\end{array}
\right).
\]
Finally in degree $3$ we have
\[
\Gamma_{+}^{(3)} = \left(
\begin{array}{ccc}
0 & 0 & q^2
\end{array}
\right)^T,\quad
\Gamma_{0}^{(3)} = \left(
\begin{array}{ccc}
0 & - q^2 & 0
\end{array}
\right)^T,\quad
\Gamma_{-}^{(3)} = \left(
\begin{array}{ccc}
q^4 & 0 & 0
\end{array}
\right)^T.
\]

We will also need the expressions for the adjoints, as in \autoref{lem:matrix-adjoint}.
In degree $0$ we have
\[
\Gamma_{+}^{(1)*} = \left(
\begin{array}{ccc}
1 & 0 & 0
\end{array}
\right)^T, \quad
\Gamma_{0}^{(1)*} = \left(
\begin{array}{ccc}
0 & 1 & 0
\end{array}
\right)^T, \quad
\Gamma_{-}^{(1)*} = \left(
\begin{array}{ccc}
0 & 0 & q^{-2}
\end{array}
\right)^T.
\]
In degree $1$ we have
\begin{gather*}
\Gamma_{+}^{(2)*} = \left(
\begin{array}{ccc}
0 & q^{-2} & 0\\
0 & 0 & q^{-2}\\
0 & 0 & 0
\end{array}
\right),\quad
\Gamma_{0}^{(2)*} = \left(
\begin{array}{ccc}
-1 & 0 & 0\\
0 & -q^{-3} (q - q^{-1}) & 0\\
0 & 0 & q^{-2}
\end{array}
\right), \\
\Gamma_{-}^{(2)*} = \left(
\begin{array}{ccc}
0 & 0 & 0\\
-q^{-4} & 0 & 0\\
0 & -q^{-2} & 0
\end{array}
\right).
\end{gather*}
Finally in degree $2$ we have
\[
\Gamma_{+}^{(3)^*} = \left(
\begin{array}{ccc}
0 & 0 & q^4
\end{array}
\right),\quad
\Gamma_{0}^{(3)^*} = \left(
\begin{array}{ccc}
0 & -q^6 & 0
\end{array}
\right),\quad
\Gamma_{-}^{(3)^*} = \left(
\begin{array}{ccc}
q^4 & 0 & 0
\end{array}
\right).
\]

\section{Commutation relations quantum Clifford}

\label{sec:comm-clifford}

In this section we will discuss the commutation relations for the quantum Clifford algebra $\mathrm{End}(\Lambda_q(\mathfrak{u}_+))$, where $\mathfrak{u}_+$ corresponds the adjoint representation of $U_q(\mathfrak{sl}(2))$.
More precisely, we will be concerned with relations between the elements $\Gamma_i = \gamma_-(w_i)$ and their adjoints. These are the elements appearing in the definition of the Dolbeault-Dirac operator.
The main result is the fact that we do not have quadratic relations among these.

\subsection{Equivariance of the construction}

Before getting into the commutation relations, let us pause for a moment to stress the equivariance of the construction of the quantum Clifford algebra. Some consequences of this property will be used in the proof below.

Recall that the quantum Clifford algebra is defined in terms of the maps $\gamma_\pm: \Lambda_q(\mathfrak{u}_\pm) \to \mathrm{End}(\Lambda_q(\mathfrak{u}_+))$.
First of all we have that $\mathfrak{u}_\pm$ are $U_q(\mathfrak{l})$-modules, which implies that $\Lambda_q(\mathfrak{u}_\pm)$ are $U_q(\mathfrak{l})$-modules, since the exterior algebras are defined in terms of module maps.
Then the quantum Clifford algebra $\mathrm{End}(\Lambda_q(\mathfrak{u}_+))$ becomes a $U_q(\mathfrak{l})$-module in a canonical way.
We denote all these actions by $\triangleright$.
The map $\gamma_+$ is clearly equivariant since it corresponds to left multiplication on $\Lambda_q(\mathfrak{u}_+)$.
On the other hand, the equivariance of the map $\gamma_-$ follows from the fact that the pairing used in its definition is invariant.
Summarizing, we find that for all $X \in U_q(\mathfrak{l})$, $v \in \Lambda_q(\mathfrak{u}_+)$ and $w \in \Lambda_q(\mathfrak{u}_-)$ we have the relations
\[
X \triangleright \gamma_+(v) = \gamma_+(X \triangleright v),\quad
X \triangleright \gamma_-(w) = \gamma_-(X \triangleright w).
\]

Next we consider the introduction of an invariant Hermitian inner product on $\Lambda_q(\mathfrak{u}_+)$.
This defines a $*$-structure on $\mathrm{End}(\Lambda_q(\mathfrak{u}_+))$.
It is compatible with the $*$-structure on $U_q(\mathfrak{l})$, coming from the compact real form of $U_q(\mathfrak{g})$.
This compatibility takes the form
\[
(X \triangleright T)^* = S(X)^* \triangleright T^*, \quad
X \in U_q(\mathfrak{l}), \ T \in \mathrm{End}(\Lambda_q(\mathfrak{u}_+)).
\]

\subsection{Absence of quadratic relations}

We are now in the position to show that we do not have quadratic commutation relations in the quantum Clifford algebra.

\begin{proposition}
\label{prop:not-quadratic}
Let $\Gamma_i = \gamma_-(w_i) \in \mathrm{End}(\Lambda_q(\mathfrak{u}_+))$, where $\mathfrak{u}_+$ is the adjoint module of $U_q(\mathfrak{sl}(2))$.
Then for $0 < q < 1$ we do not have relations of the form
\[
\Gamma_i \Gamma_j^* = \sum_{k, l} c_{i j}^{k l} \Gamma_k^* \Gamma_l, \quad i \neq j.
\]
\end{proposition}

\begin{proof}
Let us start with some general considerations for a generic $U_q(\mathfrak{l})$-module $\mathfrak{u}_-$.
Suppose we do have quadratic relations between the elements $\Gamma_{i}$ and $\Gamma_{j}^{*}$, with $i \neq j$ as above.
Then we can restrict the terms appearing on the right-hand side by using the equivariance of the map $\gamma_-$.
This can be seen as follows.
Let $\beta_i$ be the weight of the basis vector $w_i \in \mathfrak{u}_-$.
We have $K_k \triangleright w_i = q^{(\alpha_k, \beta_i)} w_i$, hence by equivariance we obtain $K_k \triangleright \gamma_-(w_i) = q^{(\alpha_k, \beta_i)} \gamma_-(w_i)$.
Similarly we find that $K_k \triangleright \gamma_-(w_i)^* = q^{-(\alpha_k, \beta_i)} \gamma_-(w_i)^*$.
This follows from the compatibility condition of the $*$-structure with the action, rewritten in the form $X \triangleright T^* = (S(X)^* \triangleright T)^*$, together with the fact that $S(K_k)^* = K_k^{-1}$.
Therefore the action on the product $\Gamma_i \Gamma_j^*$ is
\[
\begin{split}
K_k \triangleright (\Gamma_i \Gamma_j^*)
 & = (K_k \triangleright \gamma_-(w_i)) (K_k \triangleright \gamma_-(w_j)^*)\\
 & = q^{(\alpha_k, \beta_i - \beta_j)} \gamma_-(w_i) \gamma_-(w_j)^*.
\end{split}
\]
Since the elements $\{\Gamma_k^* \Gamma_l\}_{k, l}$ are linearly independent, we conclude that the term $\Gamma_k^* \Gamma_l$ can appear in the sum only when it matches the weight of the term $\Gamma_i \Gamma_j^*$.

We will now concentrate on the case of the adjoint representation.
It is enough to focus on the commutation relation between the elements $\Gamma_- = \gamma
_-(w_{-1})$ and $\Gamma_0^* = \gamma
_-(w_0)^*$.
Arguing as above, we find using equivariance that this must take the form
\[
\Gamma_- \Gamma_0^* = t \Gamma_0^* \Gamma_- + t^\prime \Gamma_+^* \Gamma_0,
\]
for some $t$ and $t^\prime$.
In particular, when acting on elements of degree $k$ it reads
\[
\Gamma_-^{(k + 1)} \Gamma_{0}^{(k + 1)*} = t \Gamma_0^{(k)*} \Gamma_-^{(k)} + t^\prime \Gamma_+^{(k)*} \Gamma_0^{(k)}.
\]
Taking into account the rescalings as in \autoref{sec:rescaling} this becomes
\begin{equation}
\label{eq:matrix-eqs}
\frac{c_{k + 1}}{c_{k}} \Gamma_{-}^{(k + 1)} \Gamma_{0}^{(k + 1)*} = \frac{c_{k}}{c_{k - 1}} (t \Gamma_{0}^{(k)*} \Gamma_{-}^{(k)} + t^{\prime} \Gamma_{+}^{(k)*} \Gamma_{0}^{(k)}).
\end{equation}
We will only consider the cases of degree $1$ and $2$, since this is enough to prove that the above identity is not satisfied for $q \neq 1$.
Let us start with $k = 1$.
Plugging in the explicit matrix expressions into \eqref{eq:matrix-eqs} we obtain the two equations
\begin{equation}
\label{eq:sol-t-tp}
\frac{c_{2}}{c_{1}} q^{-3} (q - q^{-1}) =\frac{c_{1}}{c_{0}} t^{\prime}, \quad
- \frac{c_{2}}{c_{1}} q^{-2}=\frac{c_{1}}{c_{0}} t.
\end{equation}
Clearly the coefficients $\{c_i\}_i$ must be non-zero, otherwise we would have degenerate pairings.
Hence we can use these two equations to determine the coefficients $t$ and $t^{\prime}$. Consider now the
case $k = 2$. Proceeding as before we obtain the equations
\[
-\frac{c_{3}}{c_{2}} q^{10} = \frac{c_{2}}{c_{1}} (t - t^\prime q^{-1}(q - q^{-1})), \quad
0 = \frac{c_{2}}{c_{1}} q^{-4} (t q (q - q^{-1}) + t^\prime).
\]
By the argument above, the second equation is equivalent to $t q (q - q^{-1}) + t^\prime = 0$.
Plugging in the expressions for $t$ and $t^{\prime}$ obtained from \eqref{eq:sol-t-tp}, we see that this equation has a solution only in the cases $q = \pm 1$. This concludes the proof.
\end{proof}

The absence of these quadratic relations in the quantum Clifford algebra, together with the results of  \autoref{prop:parthasarathy}, completes the proof of \autoref{thm:no-parthasarathy}.

\appendix

\section{Rescaling of pairing and inner product}
\label{sec:rescaling}

In this appendix we will discuss the effects of rescaling the dual pairing and the Hermitian inner product.
As mentioned previously this is always possible in each degree.
The aim is to have the most general setting for the discussion of commutation relations in the quantum Clifford algebra.
Although ultimately it won't play any role, we can not exclude this a priori.

We will use as before the notation $\Gamma_i = \gamma_-(w_i)$, where $\{w_i\}_i$ is a basis of $\mathfrak{u}_-$. Recall also that we write $\Gamma_i^{(k)}$ and $\Gamma_i^{(k + 1)*}$ for the restriction of $\Gamma_i$ and $\Gamma_i^*$ to elements of degree $k$.

\begin{lemma}
\label{lem:rescaling}
Under the rescalings $\langle \cdot, \cdot \rangle_k \to \lambda_k \langle\cdot, \cdot\rangle_k$ and $(\cdot, \cdot)_k \to \lambda_k^\prime (\cdot, \cdot)_k$ we have
\[
\Gamma_i^{(k)} \to \frac{\lambda_k}{\lambda_{k - 1}} \Gamma_i^{(k)},\quad
\Gamma_i^{(k + 1)*} \to \frac{\bar{\lambda}_{k + 1}}{\bar{\lambda}_k} \frac{\lambda_k^\prime}{\lambda_{k + 1}^{\prime}} \Gamma_{i}^{(k + 1)*}.
\]
\end{lemma}

\begin{proof}
Recall that $\Gamma_{i}^{(k)}:\Lambda_{q}^{k}(\mathfrak{u}_{+})\to\Lambda_{q}^{k-1}(\mathfrak{u}_{+})$
is defined by
\[
\langle z,\Gamma_{i}^{(k)}v\rangle_{k-1}=\langle z\wedge w_{i},v\rangle_{k},\quad v\in\Lambda_{q}^{k}(\mathfrak{u}_{+}),\ z\in\Lambda_{q}^{k-1}(\mathfrak{u}_{-}).
\]
Rescaling the dual pairing as $\langle\cdot,\cdot\rangle_k \to \lambda_{k}\langle\cdot,\cdot\rangle_k$ we get
\[
\Gamma_{i}^{(k)}\to\frac{\lambda_{k}}{\lambda_{k-1}}\Gamma_{i}^{(k)}.
\]

Next we consider the case of $\Gamma_i^{(k + 1)*} : \Lambda_q^k (\mathfrak{u}_+) \to \Lambda_q^{k + 1} (\mathfrak{u}_+)$.
It is defined by
\[
(\Gamma_i^{(k + 1)} v, v^\prime)_{k} = (v, \Gamma_i^{(k + 1)*} v^\prime)_{k + 1}, \quad
v \in \Lambda_q^{k + 1} (\mathfrak{u}_+), \ v^\prime \in \Lambda_{q}^{k} (\mathfrak{u}_+).
\]
Rescaling the inner product as $(\cdot,\cdot)_k \to \lambda_{k}^\prime (\cdot,\cdot)_k$ we get
\[
\Gamma_i^{(k + 1)*} \to\frac{\lambda_k^\prime}{\lambda_{k + 1}^\prime} \Gamma_i^{(k + 1)*}.
\]
Finally we combine this with the rescaling of the dual pairing. Using the fact that the inner product is conjugate-linear in the first variable we obtain
\[
\Gamma_i^{(k + 1)*} \to\frac{\bar{\lambda}_{k + 1}}{\bar{\lambda}_k} \frac{\lambda_k^\prime}{\lambda_{k + 1}^\prime} \Gamma_i^{(k + 1)*}.
\qedhere
\]
\end{proof}

In particular we need the combinations $\Gamma_{i}^{(k + 1)} \Gamma_{j}^{(k + 1)*}$
and $\Gamma_{i}^{(k)*} \Gamma_{j}^{(k)}$, acting on elements of degree $k$, which appear in quadratic relations. Upon rescaling we find
\[
\Gamma_i^{(k + 1)} \Gamma_j^{(k + 1)*} \to \frac{|\lambda_{k + 1}|^2}{|\lambda_k|^2} \frac{\lambda_k^\prime}{\lambda_{k + 1}^\prime} \Gamma_i^{(k + 1)} \Gamma_j^{(k + 1)*}
\]
and similarly for the other one. Therefore if we define $c_k = |\lambda_k|^2 / \lambda_k^\prime$ we obtain
\[
\Gamma_i^{(k + 1)} \Gamma_j^{(k + 1)*} \to \frac{c_{k + 1}}{c_k} \Gamma_i^{(k + 1)} \Gamma_j^{(k + 1)*}, \quad
\Gamma_i^{(k)*} \Gamma_j^{(k)} \to\frac{c_k}{c_{k - 1}} \Gamma_i^{(k)*} \Gamma_j^{(k)}.
\]

\vspace{3mm}

{\footnotesize
\emph{Acknowledgements}.
I would like to thank Ulrich Krähmer for several discussions on the topics of this paper.
I am supported by the European Research Council under the European Union's Seventh Framework Programme (FP/2007-2013) / ERC Grant Agreement no. 307663 (P.I.: S. Neshveyev).
}

\bigskip

\end{document}